\documentclass[10pt,oneside,a4paper]{amsart}

\usepackage{BoxOperads_Preamble}

\title{Box operads and higher Gerstenhaber brackets}	

\author{Hoang Dinh Van}
\address[Hoang Dinh Van]{Faculty of Applied Sciences, University of Technology and Education, 1 Vo Van Ngan, Thu Duc, Ho Chi Minh city, Vietnam.}
\email{hoangdv@hcmute.edu.vn}

\author{Lander Hermans}
\address[Lander Hermans]{Universiteit Antwerpen, Departement Wiskunde, Middelheimcampus,
	Middelheimlaan 1,
	2020 Antwerp, Belgium}
\email{lander.hermans@uantwerpen.be}

\author{Wendy Lowen} 
\address[Wendy Lowen]{Universiteit Antwerpen, Departement Wiskunde, Middelheimcampus,
Middelheimlaan 1,
2020 Antwerp, Belgium}
\email{wendy.lowen@uantwerpen.be}

\thanks{This project has received funding from the European Research Council (ERC) under the European Union's Horizon 2020 research and innovation programme (grant agreement No. 817762). The second named author is a predoctoral fellow of the Research Foundation - Flanders (FWO),
file number 1194422N. 
}

\newcommand{\Hor}{\mathbb{H}}

\begin{document}

\maketitle

\begin{abstract}
We introduce a symmetric operad $\boxop$ (``\emph{box-op}") which describes a certain calculus of rectangular labeled ``boxes''. Algebras over $\boxop$, which we call \emph{box operads}, have appeared under the name of fc multicategories in work by Leinster \cite{leinsterFc}. In our main result, we endow a suitable (graded, zero differential) totalisation $\boxop_{\mathrm{td}}$ with a morphism $\Linf \rightarrow \boxop_{\mathrm{td}}$. We show that $\boxop$ acts on an $\N^3$-graded enlargement of the $\N^2$-graded Gerstenhaber-Schack object $\CC_{GS}(\A)$ of a quiver $\A$ on a small category from \cite{DVL}. This action restricts to an $\Linf$-structure on $\CC_{GS}(\A)$ (with zero differential). For an element $\alpha = (m,f,c) \in \CC_{GS}^2(\A)$, the Maurer-Cartan equation holds precisely when $(\A, m, f, c)$ is a lax prestack with multiplications $m$, restrictions $f$, and twists $c$. As a consequence, the $\alpha$-twisted $\Linf$-structure on $\CC_{GS}(\A)$ controls the deformation theory of $(\A, \alpha)$ as a lax prestack.
\end{abstract}

\tableofcontents

\section{Introduction}

In the seminal work \cite{gerstenhabervoronov}, Gerstenhaber and Voronov put forth a very insightful approach to the higher structure on the Hochschild complex $\CC(A)$ of an algebra $A$. The main goal of the present paper is to initiate a similar approach to the higher structure on the Gerstenhaber-Schack complex of a prestack. In particular, the results in this paper include a complete and self contained treatment of the algebraic deformation theory of prestacks.

The Gerstenhaber-Schack complex was originally introduced and studied for presheaves of algebras by Gerstenhaber and Schack in the 1980ies motivated by the HKR decomposition and applications to schemes \cite{gerstenhaberschack} \cite{gerstenhaberschack1}, \cite{gerstenhaberschack2}. Meanwhile, their cohomology comparison theorem was generalised to prestacks by Lowen and Van den Bergh \cite{lowenvandenberghCCT}, and the ``noncommutative spaces'' in the sense of Van den Bergh, which can be described as noncommutative deformations of schemes, have become part of the general picture of homological mirror symmetry as proposed by Kontsevich in his 1994 ICM address \cite{artintatevandenbergh1990} \cite{vandenbergh2} \cite{lowenvandenberghhoch} \cite{aurouxorlov2008} \cite{kontsevich2}.

In contrast, the higher structure of the complex has remained elusive for much longer. The first positive results were obtained in special cases, for instance in \cite{FMY09}, Fr\'egier, Markl and Yau defined an $\Linf$-structure for morphisms of algebras.
In \cite{DVL}, Dinh Van and Lowen established a general approach to higher structure through homotopy transfer from the Hochschild complex of the Grothendieck construction of a prestack, building on the work of Gerstenhaber and Schack for presheaves over finite posets. However, in practice the transfer of the relevant structure is technically challenging and not necessarily illuminating. In concrete geometric cases, more direct approaches were devised, for instance by Barmeier and Fr\'egier in \cite{barmeierfregier}. The first direct operadic approach to an $\Linf$-structure for general presheaves is due to Hawkins \cite{hawkins}. This was further refined by the authors of the present paper in order to include twists into the picture \cite{vanhermanslowen2022}.

A serious drawback of these approaches is the way in which the algebraic structure of a prestack, consisting of multiplications $m$, restrictions $f$ and twists $c$, is being pulled apart in successive steps.
Morally, these components belong together in a single degree $2$ element $\alpha = (m,f,c)$ of the Gerstenhaber-Schack complex.
To remedy this, in the present paper we put forward an approach in which $\alpha$ only enters the picture in a later stage, in line with \cite{gerstenhabervoronov}.

Let us start by recalling Gerstenhaber and Voronov's approach to the algebra case from \cite{gerstenhabervoronov}.
It will be convenient to work on the underlying level of symmetric operads.  It is well known that nonsymmetric operads can be described as the algebras over an $\N$-coloured symmetric operad $\mathsf{Op}$ \cite{vanderlaan2004} \cite{hawkins}, and the first step in \cite{gerstenhabervoronov} can be rephrased in terms of the existence of a map $\mathsf{Brace} \rightarrow \Op$, where $\mathsf{Brace}$ is the symmetric operad describing brace algebras. As part of this structure, we have the prolongation
\begin{equation}\label{eqprelie}
\mathsf{PreLie} \rightarrow \ \mathsf{Brace} \rightarrow \Op
\end{equation}
which identifies a \emph{Gerstenhaber brace} on a nonsymmetric operad $\ooo$, given by
$$\phi \bullet \psi = \sum (-1)^{\epsilon}\phi(1 \otimes \psi \otimes 1),$$
for which the commutator is the Gerstenhaber Lie bracket.

Now consider an element $\alpha \in \ooo(2)$. Putting $\MC(\phi) = \phi \bullet \phi$, the Maurer-Cartan equation $$\alpha(\alpha\otimes 1) - \alpha(1 \otimes \alpha) = \MC(\alpha) = 0$$ expresses associativity of $\alpha$.
Further, the induced zero differential dgla structure can be twisted by $\alpha$ yielding a differential $d_{\alpha}(\phi) = [\alpha, \phi]$ on $\Tot(s^{-1}\ooo)$. This twisted dgla governs the deformation theory of $(\ooo, \alpha)$ through $\MC_{\alpha}(\phi) = d_{\alpha}(\phi) + \phi \bullet \phi$, a statement which takes its most familiar form for the endomorphism operad $\ooo = \CC(A)$ of a vector space $A$. The multiplication $\alpha$ also induces a dga structure on $\ooo$ through the cup product $\alpha\{\phi,\psi\}$. Taking the brace and the dga structure together, one obtains what is called a homotopy $G$-algebra in \cite{gerstenhabervoronov} (or more generally a $\Binf$-algebra \cite{getzlerjones}). This structure is a stepping stone towards a proof of the Deligne conjecture \cite{tamarkin1998}, \cite{mccluresmith2002}, \cite{bataninberger2009}. Meanwhile, the existence of a homotopy $G$-algebra structure has been proven on various Hochschild type complexes of algebro-geometric relevance, like the Hochschild complex of a dg category \cite{keller2018} and, more recently, the Tate-Hochschild complex computing singular Hochschild cohomology \cite{wang2021} (see \cite{keller2003}, \cite{chenliwang2021} for the relation between the two).

Next, we outline our approach to the case of prestacks. In analogy with the algebra case, we start with a \emph{$k$-quiver $\A$ on $\uuu$}, which corresponds to the underlying hom spaces and maps on objects of a candidate prestack on $\uuu$, but without the algebraic structure of multiplications, restrictions and twists. In a first step, we enlarge the associated $\N^2$-graded Gerstenhaber-Schack object $\CC^{p, q}_{GS}(\A)$ from \cite{DVL} to an $\N^3$-graded object $\CC_{\square}^{p,q,r}(\A)$ (see \S \ref{subparGS}).

In \S \ref{parboxop}, we introduce an $\N^3$-coloured symmetric operad $\boxop$ which is seen to act on $\CC^{p,q,r}_{\square}(\A)$ from its definition in terms of generators and relations. More precisely, colours $(p,q,r)$ are pictured as labeled boxes
$$ \scalebox{0.7}{$\input{fcRectangle.tikz}$} $$
and the generators of $\boxop$ are depicted as stackings of labeled boxes
$$ \scalebox{0.7}{$\input{2Level_stack.tikz}$}   $$
Such a generator acts on $\CC^{p,q,r}_{\square}(\A)$ by filling the boxes with linear maps of the corresponding colour, and then composing:
$$ \scalebox{0.7}{$\input{boxoperadGS_simple.tikz}$}$$

Algebras over $\boxop$ will be called \emph{box operads}. They are an instance of multicategories over a monad in the sense of Burroni \cite{burroni1971}, and have been considered by Leinster under the name of fc multicategories \cite{leinsterFc} \cite{Leinster2004}. They have also been called virtual double categories, see for instance \cite{CruttwellSchulman2010}, \cite{Koudenburg2020}.

Our main result is the existence of a canonical $\Linf$-structure on a (suitably totalised) linear box operad in \S \ref{parbrackets}. To this end, we define a morphism of symmetric dg operads
\begin{equation}\label{eqmain}
\Linf \rightarrow \boxop_{\mathrm{td}}
\end{equation}
for
a (graded, zero differential) totalisation $\boxop_{\mathrm{td}}$ of $k\boxop$ (Theorem \ref{Linf_theorem}).

In the final section \S \ref{parGS}, we obtain the desired application to prestacks by showing that the Gerstenhaber-Schack object $\Tot (s^{-1}\CC_{GS}(\A))$ is an $\Linf$-subalgebra of the totalised linear box operad $\Tot (s^{-1}_{\thin}\CC_{\square}(\A))$ (Proposition \ref{propGSLinf}).
We further show that a Maurer-Cartan element $\alpha = (m,f,c) \in \Tot (s^{-1}\CC_{GS}(\A))$ corresponds precisely to a lax prestack structure $(m,f,c)$ on $\A$. As a consequence, using the general machinery of twisting \cite{merkulov2000} \cite{dotsenko_shadrin_vallette_2023} (over $k$ a field of characteristic $0$), we obtain that the deformation theory of a (lax) prestack $(\A, \alpha)$ is governed by the $\alpha$-twisted $\Linf$-algebra $\Tot_{\alpha} (s^{-1}\CC_{GS}(\A))$.

The definition of the $\Linf$ components in $\boxop_{\mathrm{td}}$ from \eqref{eqmain} makes fundamental use of a combinatorial description of $\boxop$ in terms of stackings of boxes which is established in \S \ref{parboxop}.
%$$S = \quad \scalebox{0.85}{$\tikzfig{example_stacking_simple}$}\quad .$$
This makes use of two underlying structures of a stacking $S$:  the \emph{vertical composite tree} $\mathbb{V}_S$ and the \emph{horizontal composite graph} $\mathbb{H}_S$, for instance % (see \eqref{verticalhorizontal}).
$$S = \quad \scalebox{0.85}{$\input{example_stacking_simple.tikz}$} \quad ,\quad  \V_S = \quad \scalebox{0.85}{$\input{example_tree_simple.tikz}$}  \quad \text{ and } \quad  \Hh_S = \quad \scalebox{0.85}{$\input{example_graph_simple.tikz}$} .$$
 These two latter structures have the boxes of the stacking as vertices, the former being built from vertical adjacencies and the latter from horizontal ones. Correspondingly, in Theorem \ref{thmboxopSUBoppro} we give a description of $\boxop$ as a suboperad of $\Op \times \Pro$, where $\Pro$ is the $\N^2$-coloured symmetric operad encoding pros or asymmetric props.

We end this introduction by discussing some work in progress.

In analogy with the algebra case in \cite{gerstenhabervoronov}, one can go on to consider box operads with multiplication. The $\Linf$-structure from \eqref{eqmain} has an underlying $\Linf$-admissible structure taking the form of higher Gerstenhaber braces. In \cite{hermanslowenhigherstructure}, we extend this definition to the structure of a $\Ginf$-algebra in the sense of \cite{voronov}. Note that such a structure is known to exist on the Gerstenhaber-Schack complex of a prestack by homotopy transfer \cite{DVL} \cite{bergermoerdijk2003} \cite{getzlerjones}. The comparison of the various descriptions of $\Linf$-structures currently known to exist on this complex is also relegated to \cite{hermanslowenhigherstructure}.
Further, in future work we will extend our approach to double categories, $n$-fold categories and higher dimensional analogues of box operads. Recently, it was pointed out to us by Ralph Kaufmann that concepts in \cite{kaufmannmonaco2022} may be related to box operads, opening a potential avenue for further generalisations.

Finally, the theory of operads provides a general and powerful approach to deformation theory via the definition of operadic deformation complexes \cite{markl2010}. This approach will be investigated in the setup of box operads in \cite{hermanskoszulduality}, including the development of Koszul duality and the construction of cofibrant models.

\medskip 

\noindent \emph{Acknowledgement.} 

This manuscript was completed during a research visit of the second named author at Paris 13. The second named author thanks Bruno Vallette for the hospitality during this visit, and for numerous valuable discussions and explanations, in particular on the topics of homotopy transfer and twisting.

The authors are grateful to the organisers of the Nouveau S\'eminaire ALPE in March 2023 at IMAG, Universit\'e de Montpellier, of the Mid-term workshop of the ANR HighAGT at IRMA, Universit\'e de Strasbourgh in May-June 2023, and of the 2023 Thematic trimester program ``Higher Structures in Geometry and Mathematical Physics'' at IHP, Paris, for the various opportunities to present and discuss this work.

\medskip

\noindent \emph{Conventions.}

We work over a commutative ground ring $k$.

We will consider both ``symmetric'' and ``nonsymmetric'' operads in this paper. While we will typically add these adjectives when an operad is first introduced, we will omit them most of the time.
Similarly, several of the operads we use will be ``coloured'' operads, but we will not always stress this.

We use interchangeably the two equivalent compositional structures $- \circ_i -$ and $- \circ (-,\ldots,-)$ for operads. Furthermore, we write $x \circ (- ,x_2,\ldots,x_n)$ for the corresponding composition where the second input variable is left out.

For $n\in \N$, put $[n] := \{0,\ldots,n\}$ and $\lh n \rh := \{ 1,\ldots,n \}$.

\section{Box operads}\label{parboxop}

In this section, we introduce and study the $\N^3$-coloured symmetric operad $\boxop$ (pronounced as ``box-op''), which models natural ways of stacking $(p,q,r)$-labelled boxes:

\begin{equation}\label{genx}
\scalebox{0.7}{$\input{2Level_stack.tikz}$}
\end{equation}

This operad is seen to act on an $\N^3$-graded enlargement of the $\N^2$-graded Gerstenhaber-Schack object of a prestack from \cite{DVL}, as will be detailed in \S \ref{parGS}. Algebras over $\boxop$ will be called \emph{box operads} and have appeared in the literature under different names.
More precisely, box operads are an instance of multicategories over a monad in the sense of Burroni \cite{burroni1971} and have been considered by Leinster under the name of fc multicategories \cite{leinsterFc, Leinster2004}. They have also been called virtual double categories, see for instance \cite{CruttwellSchulman2010}, \cite{Koudenburg2020}.

The operad $\boxop$ can be easily defined in terms of generators \eqref{genx} and associativity relations (Definition \ref{defgenFc}), and the main goal of this section is to provide a combinatorial description in terms of more general stackings of boxes (Theorem \ref{ThmBoxopStackings}). This description makes use of two underlying structures of a stacking:  the \emph{vertical composite tree} and the \emph{horizontal composite graph}. Both these structures have the boxes of the stacking as vertices, the former being built from vertical adjacencies and the latter from horizontal ones (see pictures in Example \ref{exTreeGraph}).

Before investigating $\boxop$, we first recall the operads capturing these two underlying structures. First, in \S \ref{subparop}, we recall the symmetric $\N$-couloured operad $\Op$ whose algebras are nonsymmetric operads, and whose combinatorial elements are planar trees \cite{vanderlaan2004}. Secondly, in \S \ref{subparpro}, we recall the symmetric $\N^2$-coloured operad $\Pro$ whose algebras are pros or asymmetric props, and whose combinatorial elements are planar graphs \cite{yauJohnson2015}.

In \S \ref{subparboxop} we define $\boxop$ in terms of generators and relations, as well as two morphisms of (coloured) symmetric operads with the following prescriptions on colours:
\begin{equation}\label{V}
        \V : \boxop \rightarrow \Op: (p,q,r) \mapsto q
\end{equation}

\begin{equation}\label{H}
        \Hor : \boxop \rightarrow \Pro: (p,q,r) \mapsto (p,r)
\end{equation}
Finally, in \ref{subparstack} we give a combinatorial description of $\boxop$ and realise it as a suboperad of $\Op \times \Pro$. Based upon this description, the higher morphisms
$\V$ and $\Hor$ correspond precisely to taking the vertical composite tree, respectively the horizontal composite graph, of a stacking of boxes.

\medskip

%\noindent \emph{Conventions.}

%We work over a commutative ground ring $k$.

%Operads can either have a single colour or a set of colours, in which case we call them ``coloured operads''. Furthermore, they can come equipped with a symmetric action, which we call ``symmetric operads''. When such an action is not taken into ac count, we emphasize this by writing ``nonsymmetric''. In general, we omit these adjectives when the correct notion of ``operad'' is clear from context. We add the adjectives ``nonsymmetric'', ``symmetric'' and `` coloured'' when we want to emphasize the extra structure (or lack thereof).

%For all notions of operad, we use interchangeably their two equivalent compositional structures $- \circ_i -$ and $- \circ (-,\ldots,-)$. Furthermore, we write $x \circ (- ,x_2,\ldots,x_n)$ for the corresponding composition where the second input variable is left out.

%For $n\in \N$, put $[n] := \{0,\ldots,n\}$ and $\lh n \rh := \{ 1,\ldots,n \}$.

%\subsection{Box operads}\label{subparboxop}

\subsection{The operad $\Op$}\label{subparop}

The $\N$-coloured symmetric operad $\Op$ is generated by a unit $\eta \in \Op\left( ; 1 \right)$ and, for every $q_1,\ldots,q_n\in \N$, a generator
$$C_{\nth{q}} \in \Op (n,q_1,\ldots,q_n; \sum_{i=1}^n q_i)$$
For matching indices, we require them to satisfy the relations
\begin{enumerate}
\item $ 
 C_{\underline{q}^1,\ldots,\underline{q}^n} \circ_1 C_{\underline{m}} = \left(  C_{\sum_{j=1}^{m_1} q^1_j,\ldots, \sum_{j=1}^{m_n} q^n_j}  \circ (- ,C_{\underline{q}^1},\ldots,C_{\underline{q}^n})\right)^{\si}$ where $\si$ is the usual permutation shuffling 
\begin{equation}\label{shuffleOp}
(x;(x_1;x^1_1,\ldots,x^1_{m_1}),\ldots,(x_n;x^n_1,\ldots,x^n_{m_n}) ) \leadsto ((x;x_1,\ldots,x_n);x^1_1,\ldots,x^n_{m_n}),
\end{equation}
\item $C_{1,\ldots,1} \circ (-,\eta,\ldots,\eta) = 1_n = C_1 \circ_1 \eta$.
\end{enumerate}
where $1_n \in \Op(n;n)$ is the operadic unit. We often omit the indices and simply write $C_n$. 
%We can draw these generators as 
%$$C_{q_1,\ldots,q_n} = \quad \scalebox{0.75}{$\tikzfig{corolla_tree_full}$}  \quad \text{ and } \quad \eta = \quad \scalebox{0.85}{$\tikzfig{unit_tree}$} $$
%where the numbers on top of the leaves represent the numbering of the inputs.

An $\Op$-algebra corresponds to an operad.

\subsubsection{The operad of trees with unit}\label{subsubpartrees}
We provide a generator-free description of $\Op$. Following the presentation from \cite[\S 2.1]{hawkins}, a (planar) tree $T$ consists of a quadruple $T = (V,E,\tre,<)$ where $V$ is a set of vertices $V$, $E \sub V \times V$ is a non-empty set of edges and $\tre,<$ are two partial orders on $V$ which we interpret as 
\begin{align*}
below\; &<\; above\\
left \; &\tre \; right
\end{align*}
For vertices $u,v \in V$, if $(u,v) \in E_T$, we call \emph{$u$ the parent of $v$} and \emph{$v$ the child of $u$}. If a vertex has a child, we call it \emph{internal}, and otherwise a \emph{leaf}. A tree $T$ induces a \emph{total order} $\downarrow$ on the set of vertices $V$: $i \downarrow j $ if $i$ is left of $j$ or $i$ is below $j$.

For $\nth{q},q \in \N$, let $\Tree(\nth{q};q)$ be the set of trees $T$ on vertex set $\lh n \rh$ \emph{full with respect to $\nth{q},q$}, that is, such that each internal vertex $i$ has exactly $q_i$ children and $q = \sum_{v \text{ leaf}} q_v$. Note that for such trees holds $q-1= \sum_{i=1}^n (q_i-1)$. $T$ is in \emph{standard order} if the induced total order $\downarrow$ on the vertices agrees with the natural order of $\lh n \rh$.

The generators $C_{\nth{q}}$ of $\Op$ correspond to the $n$-corollas, that is, the trees
$$ \scalebox{0.75}{$\input{corolla_tree_full.tikz}$} $$ 
Full trees form an $\N$-coloured operad where composition is given by substitution of trees in vertices. The symmetric action is given by permuting the vertices. In particular, it corresponds to the suboperad of $\Op$ generated by the $n$-corollas $C_n$.

The operad $\Op$ then corresponds to adding the unit $\eta\in \Op(;1)$ as a formal element to the operad of full trees with the following relations
\begin{enumerate}
\item let $i$ be an internal vertex of a full tree $T$ with exactly one child $j$ and let $T'$ be the full tree obtained from $T'$ by contracting the subtree spanned by vertices $i$ and $j$, then we have $T \circ_i \eta = T'$.
\item let $j_1 \tre \ldots \tre j_{q_i}$ be the set of children of an internal vertex $i$ in full tree $T$ with corresponding number of inputs $q_{j_s}$. Assume for simplicity $j_1 < \ldots < j_{q_i}$ as numbers. Let $T'$ be the full tree obtained from $T$ by contracting the subtree spanned by the vertices $i,j_1,\ldots,j_{q_i}$, then we have 
$T \circ_{j_{q_i}} \eta \circ_{j_{q_i-1}} \ldots \circ_{j_1} \eta =T'$. 
%\item for $i_1<\ldots <i_n$ the set of leaves of a full tree $T$  and let $T'$ be the full tree $T$ with its leaves omitted, then we have $T \circ_{i_n} \eta \circ_{i_{n-1}} \ldots \circ_{i_1} \eta =T'$.
\end{enumerate}

\subsection{The operad $\Pro$}\label{subparpro}

Let $\Pro$ be the $\N^2$-coloured symmetric operad generated by the elements
\begin{itemize}[itemsep=4mm]
\item $M_{p,r,t} \in \Pro\left( (p,r), (r,t); (p,t) \right)$ for every $p,r,t\in \N$,
\item $P_{p,r}^{p',r'} \in \Pro\left( (p,r), (p',r'); (p+p', r+r') \right)$ for every $p,p',r,r' \in \N$,
\item $\eta_P \in \Pro\left( ;(0,0)\right)$ and $\eta_M \in \Pro\left( ; (1,1) \right)$
\end{itemize}
satisfying the relations
\begin{enumerate}[leftmargin=5.5mm, itemsep=4mm]
\item $M_{p,t,s} \circ_1 M_{p,r,t} = M_{p,r,s} \circ_2 M_{r,t,s}$,
\item $P^{p'',r''}_{p+p',r+r'} \circ_1 P^{p',r'}_{p,r} = P^{p'+p'', r'+r''}_{p,r} \circ_2 P^{p'',r''}_{p',r'}$,
\item $M_{p+p',r+r',t+t'} \circ_2 P_{r,t}^{r',t'} \circ_1 P_{p,r}^{p',r'} = \left( P_{p,t}^{p',t'} \circ_2 M_{p',r',t'} \circ_1 M_{p,r,t} \right)^{(23)}$,
\item $P_{p,r}^{0,0} \circ_2 \eta_P = 1_{p,r} = P_{0,0}^{p,r} \circ_1 \eta_P$,
\item $M_{p,r,r} \circ_2  \eta_M^r = 1_{p,r} = M_{p,p,r} \circ_1 \eta_M^p$
\end{enumerate}
where $1_{p,r} \in \Pro((p,r);(p,r))$ is the operadic unit and $ \eta_M^r := ( P_{r-1,r-1}^{1,1} \circ_1 \ldots \circ_1 P_{1,1}^{1,1} ) \circ ( \eta_M,\ldots,\eta_M) $.

A $\Pro$-algebra is a pro or an asymmetric prop.
\subsubsection{The operad of planar graphs}
We provide a generator-free description of $\Pro$. Let $\lh n \rh$ denote the set $\{1,\ldots,n\}$. Note $\lh 0 \rh $ is the empty set. A (directed, numbered) graph $G= (\lh n \rh, E,f)$ with respect to the pairs of natural numbers $(p_1,r_1),\ldots,(p_n,r_n);(p,r)$, consists of 
\begin{itemize}
\item the set of vertices $V:= \lh n \rh$,
\item for every $a,b \in V$ a collection of partial functions $E(a,b) \sub \lh r_b \rh \times \lh p_a \rh$ connecting outputs of $b$ with inputs of $a$,
\item two injective functions $f_{in}: \lh p \rh \longrightarrow \coprod_{a \in V} \lh p_a \rh \coprod \lh r \rh$ and $f_{out}: \lh r \rh \longrightarrow \coprod_{a\in V} \lh r_a \rh \coprod \lh p \rh$ numbering the inputs and outputs respectively of the graph $G$ as a whole,
\end{itemize}
such that 
\begin{enumerate}
\item each input $i \in \lh p_a \rh$ is either \emph{plugged} by a unique output $j\in \lh r_b \rh$, that is, $(j,i) \in E(a,b)$, or is an \emph{input of $G$}, that is, $(i \in a) \in \Image(f_{in})$,
\item each output $j \in \lh r_b \rh$ is either \emph{plugging} a unique input $i\in \lh p_a \rh$, that is, $(j,i) \in E(a,b)$, or is an \emph{output of $G$}, that is, $(j \in b) \in \Image(f_{out})$,
\item if an input $i$ of $G$ is connected to an output $j$ of $G$, i.e. $f_{in}(i) = (j \in \lh r \rh)$, the output $j$ of $G$ is connected to the input $i$ of $G$, i.e. $f_{out}(j) = (i \in \lh p \rh)$, and vice versa. 
\end{enumerate}
Note that for its inputs and outputs holds $ r-p = \sum_{a \in V} (r_a-p_a)$. 

We embed a graph $G$ in the plane as follows: we draw $G$ directed from left (inputs) to right (outputs) and number both inputs and outputs from top to bottom, as follows
\begin{itemize}
\item for every vertex $a$ we draw a disk labelled $a$ with $p_a$ number of inputs and $r_a$ number of outputs
$$\input{spider.tikz}$$
\item for every element $(j,i) \in E(a,b)$, we draw an edge from the output $j$ of vertex $b$ to the input $i$ of vertex $a$.
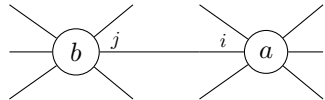
$$\input{spider_edge.tikz}$$
\item we vertically draw left of the graph inputs $1$ (top) to $p$ (bottom) and draw an edge from $i \in \lh p \rh$ to the input $f_{in}(i) = j \in \lh p_a \rh$ of vertex $a$,
\item we vertically draw right of the graph outputs $1$ (top) to $p$ (bottom) and draw an edge from the output $f_{out}(i) \in  \lh r_a \rh$ of vertex $a$ to the  output $i \in \lh r \rh$ of the graph.
\item if $f_{in}(i) = j \in \lh r \rh$ and $f_{out}(j) = i \in \lh p \rh$, then we draw an edge from input $i$ to output $j$.
\end{itemize}
%We consider an example.
%\begin{vb}
%Let the graph $G$ be defined as
%\begin{align*}
%V = \{ 1,2,3\} \qquad  \tikzfig{edges}\qquad  \tikzfig{fin} \qquad \tikzfig{fout}
%\end{align*}
%which we can draw as 
%$$\tikzfig{exGraph}$$
%Note that this is not planar as we have two edges crossing.
%\end{vb}

A graph $G$ is \emph{planar} if its embedding in the plane described above is possible without crossing edges.

The generators $M_{p,r,t}$ and $P_{p,r}^{p',r'}$ then correspond to the following respective graphs
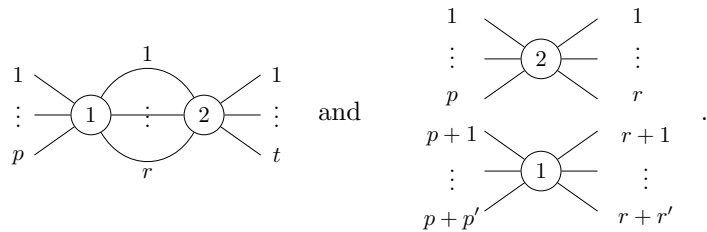
\begin{equation*}\label{drawings_generators_pro}  \scalebox{0.85}{$\input{Cprt.tikz}$} \quad \text{ and } \quad   \quad \scalebox{0.85}{$\input{Box.tikz}$}\quad.
\end{equation*}
The units $\eta_P$ and $\eta_M$ correspond respectively to the unique graphs without vertices with respectively exactly $0$ an $1$ input and output.

The operad $\Pro$ corresponds to the $\N^3$-coloured operad spanned by the planar graphs above whose composition is given by substitution of a planar graph in a vertex of a planar graph. The $\Ss$-action is given by permuting the vertices of the graph.

\subsection{The operad $\boxop$}\label{subparboxop}

%We define our main operad of interest.

%\medskip
%
%\noindent \emph{Notations.}
%In this section, specifically for the letters $q$ and $p$, let $\underline{q}$ and $\underline{p}$ denote respectively sequences $(q_1,\ldots,q_n)$ and $(p_0,\ldots,p_n)$ in $\N$ for some $n\geq 0$. Further, write $\underline{\widehat{p}}$ for the associated sequence $(p_1,\ldots,p_n)$.
We introduce our main operad of interest.

\begin{mydef}\label{defgenFc}
Let $\boxop$ be the $\N^3$-coloured operad generated by  
\begin{itemize}
\item for every $n,\nth{q}, p_0,\ldots,p_n,p ,r \in \N$ an element
$$ C^{p_0,\ldots,p_n;p,r}_{q_1,\ldots,q_n} \in \boxop\left( (p,n-1,r),(p_0,q_1,p_1),\ldots,(p_{n-1},q_{n},p_n); (p+p_0,\sum_{i=1}^{n}q_i,r+p_n)\right)$$
%\item for every $p,r \in \N$ a horizontal unit
%$$ \eta_h^{p,r} \in \boxop\left( (p,0,r); (p+1,0,r+1) \right),$$
\item a unit $\eta \in \boxop\left( ; \left(0,1,0\right) \right)$
\end{itemize}
%For clarity, we often omit the indices and simply write $C_n$.\newline
For \emph{matching} colours
$$(p,n,r),(p_0,m_1,p_1),\ldots,(p_{n-1},m_n,p_n),(p^1_0,q^1_1,p^1_1),\ldots,(p^1_{q_1-1},q^1_{m_1},p^1_{m_n}),\ldots,(p^n_0,q^n_1,p ^n_1),\ldots, (p^n_{m_n-1},q^n_{m_n},p^n_{m_n}),$$
that is, such that for $0 < i < n$ holds
$$p^i_{m_i} = p^{i+1}_0,$$
we require associativity and unit relations
\begin{enumerate}
\item $ C_{\underline{q}^1,\ldots,\underline{q}^n}^{\underline{p}^1,\underline{\hat{p}}^2,\ldots,\underline{\hat{p}}^n} \circ C_{\underline{m}}^{\underline{p};p,r} = \left( C^{p_0+p^1_0,\ldots,p_{n-1}+p^{n}_0,p_n+p^n_{m_n};p,r}_{\sum_{j=1}^{m_1} q^1_j,\ldots, \sum_{j=1}^{m_n} q^n_j} \circ ( - , C_{\underline{q}^1}^{\underline{p}^1;p_0,p_1} , \ldots, C_{\underline{q}^n}^{\underline{p}^n;p_{n-1},p_n}) \right)^{\si} $ \label{eqboxopAssoc}
\end{enumerate}
where $\si$ is the usual shuffle \eqref{shuffleOp} and $\underline{\hat{p}}^i := p^i_1,\ldots,p^i_{m_i}$ the (possibly empty) sequence omitting the first element, 
%usual permutation shuffling 
%$$(x;(x_1;x^2_1,\ldots,x^2_{m_1}),\ldots,(x_n;x^n_1,\ldots,x^n_{m_n}) ) \leadsto ((x;x_1,\ldots,x_n);x^2_1,\ldots,x^n_{m_n-1}),$$.
%, and $\eta_h^{a_i} := \underbrace{\eta_h \circ \ldots \circ \eta_h}_{a_i\text{-times}}$, where $\eta_h^0 := 1$ the operadic unit.\newline
%We further require a unitality relation
\begin{enumerate}[resume]
%\item $\eta_h \circ_1 C_n = C_n \circ(-,\eta_h,\ldots,\eta_h)$,
\item $C_n^{p,r;0,0} \circ_1 \eta = 1 =  C_{1,\ldots,1}^{0,\ldots,0;p,r} \circ (-,\eta,\ldots,\eta) $,
\end{enumerate}
where $1$ is the operadic identity.
\end{mydef}
%\begin{opm}\label{remmatchingindices}
%We make explicit when colours are \emph{matching} for the associativity condition. The condition is empty if $m_1=\ldots=m_n =0$. Assume this is not the case, then define two (possibly infinite) numbers 
%$$ l_i := \max(\{k < i \; | \; m_k > 0 \} \cup \{- \infty	\}) \text{ and } r_i :=\min( \{k\geq i \; | \; m_k >0 \} \cup \{ +\infty	\} )$$
%for every $1 \leq  i \leq n$. Note that if $m_i>0$, then $r_i = i$. The condition can then be stated as
%$$ \text{for }1 \leq i < n\text{, if both }l_i\text{ and }r_i\text{ are finite, then }p^{l_i}_{m_{l_i}} = p^{r_i}_{0}.$$
%For matching indices, we further define 
%$$a_i = \begin{cases} p_0^0 &\text{ if }  m_1 = \ldots = m_n = 0 \\
%0 & \text{ if } m_i>0 \\
%p^{j}_{m_j}& \text{ if }  j = \max\{k< i \; | \; m_k > 0 \}\geq 1 \\
%p^{j'}_{0} &\text{ if } j' = \min \{k> i \; | \; m_k >0 \} \leq n \end{cases}.$$
%\end{opm}

An $\boxop$-algebra is a called a \emph{box operad}. Let us spell this out. 

%\begin{lemma}
%For $\eta_h^{p,r} :=C^{1;p,r}$, holds  $C^{p_0;p,r} = \eta^{p+p_0,r+p_0-1}_h  \circ_1 \ldots \circ_1 \eta^{p+1,r+1}_h$. Moreover, if $a_i>0$ in \eqref{eqboxopAssoc}, then 
%$C^{\underline{p}^i;p_{i-1},p_i}_{\underline{q}^i} = C^{a_i;p_{i-1},p_i}$.
%The relations
%\begin{enumerate}
%\item $C^{p_0;p,r} = \eta^{p+p_0,r+p_0-1}_h  \circ_1 \ldots \circ_1 \eta^{p+1,r+1}_h$,
%\item If $a_i>0$ in \eqref{eqboxopAssoc}, then 
%$C^{\underline{p}^i;p_{i-1},p_i}_{\underline{q}^i} = C^{a_i;p_{i-1},p_i}$. 
% \end{enumerate}
%uniquely determine the generators $\eta_h^{p,r}$ in $\boxop$.
%\end{lemma}
%We make a distinction between the generators $C_0$ and $C_n$ for $n\geq 1$. The generators 

\begin{mydef}\label{defboxoperad}
A \emph{box operad} $\bbb$ consists of a collection of $k$-modules $(\bbb(p,q,r))_{p,q,r\in \N}$, composition maps
$$\mu: \bbb(p,n,r) \otimes \bigotimes_{i=1}^n\bbb(p_{i-1},q_i,p_{i}) \longrightarrow \bbb(p+ p_0, \sum_{i=1}^n q_i, r + p_{n}),$$
%horizontal unit maps
%$$\eta_h: \ppp(p,0,r) \longrightarrow \ppp(p+1,0,r+1),$$
and a unit
$$\eta: k \longrightarrow \bbb(0,1,0),$$ 
satisfying associativity and unit axioms: for $x,x_1,\ldots,x_n, x_1^1,\ldots,x^n_{k_n}\in \bbb$ in \emph{matching} arities (see Definition \ref{defgenFc}), we have
\begin{enumerate}
\item \label{boxop_assoc} $\mu( x; \mu(x_1; \ofromto{x^1}{q_1}) \otimes \ldots \otimes \mu(x_n; \ofromto{x^n}{q_n})) = \mu(\mu(x;x_1 \otimes \ldots \otimes x_n); x_1^1 \otimes \ldots \otimes x_{q_n}^n)$,
\item $\mu(x; 1 \otimes \ldots \otimes 1) = x = \mu(1;x)$,
%\item $\mu(x; \eta_h(x_1) \otimes \ldots \otimes \eta_h(x_n) ) = \eta_h(\mu(x;\onth{x}))$
\end{enumerate}
%for suitable natural numbers $a_i \in \N$ (see remark \ref{remai}) and $\eta_h^{a_i} := \underbrace{\eta_h \circ \ldots \circ \eta_h}_{a_i\text{-times}}$.
\end{mydef}

For pictorial representations, we refer to $\S \ref{subparstack}$ (notably pictures \eqref{drawingGeneratorsBoxop} and \eqref{drawingAssocStackings}) in which a combinatorial description of $\square \mathsf{p}$ is given.

%intuition, we refer to their combinatorial counterpart provided in $\S \ref{subparstack}$ (notably pictures \eqref{drawingGeneratorsBoxop} and \eqref{drawingAssocStackings}).

%
%We provide intuition through drawings. An element in arity $(p,q,r)$ corresponds to a labelled box 
%\begin{equation}\label{drawingBox}
%
%\end{equation}
%The composition models composing boxes by stacking
%\begin{equation}\label{drawingStacking}
%
%\end{equation}
%The associativity condition 

\subsubsection{The morphism $\V: \boxop \longrightarrow \Op$}

\begin{prop}\label{morphismtrees}
We have a morphism of coloured symmetric operads
$$\V:\boxop \longrightarrow \Op$$
sending the colour $(p,q,r)$ to $q$, defined on generators as
$$\V(C^{p_0,\ldots,p_n;p,r}_{q_1,\ldots,q_n}) := C_{q_1,\ldots,q_n}\text{ and } \V(\eta) := \eta .$$
\end{prop}
\begin{proof}
The relations in $\boxop$ are readily verified.
\end{proof}

\subsubsection{The morphism $\Hor: \boxop \longrightarrow \Pro$}

\begin{prop}\label{morphismgraphs}
We have a morphism of coloured symmetric operads 
$$\Hor: \boxop \longrightarrow \Pro$$
sending the colour $(p,q,r)$ to $(p,r)$, defined on generators as 
\begin{align*}
\eta &\longmapsto \eta_P   \\
C^{p_0,\ldots,p_n;p,r}_{\fromto{q}{n}} &\longmapsto B_{p,r}^{p_0,p_n} \circ_2 (D_{p_0,\ldots,p_n}) 
\end{align*}
where 
$$D_{p_0,\ldots,p_n} :=  D_{p_1,p_{n-1},p_n} \circ_1 D_{p_1, p_{n-2}, p_{n-1}} \circ_1 \ldots \circ_1 D_{p_1,p_2,p_3}$$
\end{prop}
\begin{proof}
The relations in $\boxop$ are readily verified.
\end{proof}

\subsection{Stackings of boxes}\label{subparstack}

%We formalize the suitable notion of stackings of boxes and give a combinatorial description of the operad $\boxop$. Similar to our description of $\Op$ via trees, we proceed in two steps. First, we define a $\N^3$-coloured operad of stackings $Stack$ corresponding to the suboperad of $\boxop$ generated by the $n$-corollas $C_n$ and the horizontal unit maps $\eta_h$. Second, we realize $\boxop$ by adding a unit $\eta$ to the operad $\Stack$.

\subsubsection{The operad of stackings}

In this section, we introduce ``stackings of boxes''. Roughly speaking, these are trees with ``matching'' boxes as vertices.

%In this section, we call trees ``stackings'' and picture their vertices by boxes.

\begin{mydef}
Consider
\begin{enumerate}
\item colours $(p_{1},q_{1},r_1),\ldots,(p_{n},q_{n},r_n)\in \N^3$,
\item a tree $S\in \Tree(\nth{q};q)$,
\item a collection of natural numbers $(s_i )_{i \text{ leaf of } S}$.
\end{enumerate}
We call the datum of (1) and (2) a \emph{stacking $S$ of the boxes $(p_{1},q_{1},r_1),\ldots,(p_{n},q_{n},r_n)$}, and we will often refer to this simply as a \emph{stacking} $S$, with the boxes being understood as part of the data, sitting at the vertices.
Given a stacking $S$, we call (3) a \emph{collection of start values} for $S$.
We inductively define for every vertex $i$ the \emph{left depth $dL(i)$} and the \emph{right depth $dR(i)$} as follows:
 \begin{enumerate}
 \item If $i$ is a leaf, then set 
 $$dL(i) := s_i + p_i \quad \text{ and } \quad dR(i) := s_i + r_i $$
 \item If $i$ is not a leaf, with left-most child $b$ and right-most child $c$, then set 
 $$dL(i) := dL(b) + p_i    \quad \text{ and } \quad dR(i) := dR(c) + r_i$$
 \end{enumerate}
 $$\scalebox{0.8}{$\input{Depth.tikz}$}$$
 \end{mydef}
% We consider a concrete example.
% \begin{vb}\label{exDepth}
% Given the data 
% \begin{alignat*}{3}
% &(p_1,q_1,r_1) = (2,3,1) \quad &&(p_4,q_4,r_4) = (1,2,1) \quad &&(p_7,q_7,r_7) = (1,1,1) \\
% &(p_2,q_2,r_2) = (0,2,0)\quad &&(p_5,q_5,r_5) = (3,1,1)\quad &&(p_8,q_8,r_8) = (1,1,1) \\
% &(p_3,q_3,r_) = (2,0,3)\quad &&(p_6,q_6,r_6) = (3,1,0)\quad &&
% \end{alignat*}
% consider the following stacking of boxes
% $$\scalebox{1}{$\tikzfig{exDepth}$} \quad \text{ which has the tree structre } \quad  \scalebox{1}{$\tikzfig{exDepth_tree}$}$$
% and corresponds to the start data
%$$ s_3 = 1, \quad s_5 = 0, \quad s_8 = 1, \quad s_7 =0, \quad s_6 = 0$$
%Then we have the following depths
%\begin{alignat*}{4}
%&dL(1) = 5,\quad &&dR(1)= 1, \quad  &&dL(5) = 3,\quad &&dR(5)= 1 \\
%&dL(2) = 3, &&dR(2)= 3, &&dL(6) = 3, &&dR(6)= 0 \\
%&dL(3) = 3, &&dR(3)= 4, &&dL(7) = 1, &&dR(7)= 1 \\
%&dL(4) = 4, &&dR(4)= 3, &&dL(8) = 2, &&dR(8)= 2 
%%&dL(5) = 3, &dR(5)= 1 \\
%%&dL(6) = 3, &dR(6)= 0 \\
%%&dL(7) = 1, &dR(7)= 1 \\
%%&dL(8) = 2, &dR(8)= 2 
%\end{alignat*}
%As we can see from the drawings, the result is not a box. This is reflected by the depths: both boxes $2$ and $7$ are children of box $1$, but their depths do not agree as $dL(7) = 1 \neq 3 = dR(2)$.
% 
%A second observation is that the horizontal unit map $\eta_h$ for fc operads is reflected by the start value $s_3= 1$. However, we do not want a non-zero start value $s_i$ when $q_i$ is non-zero as this has no counterpart for fc operads.
% \end{vb}
%We incorporate the above observations to determine admissibility.
\begin{mydef}\label{defStacking}
A stacking $S\in \Tree(\nth{q};q)$ of boxes $(p_{1},q_{1},r_1),\ldots,(p_{n},q_{n},r_n); (p,q,r)$ is \emph{admissible} if there exists a collection of start values $(s_i)_i$ such that
\begin{enumerate}
\item the start value $s_i =0$ if $q_i >0$,
\item for an internal vertex $b$ the adjacent depths of its children $b_{1} \tre \ldots \tre b_{q_{a}}$ agree, that is, $dR(b_{i}) = dL(b_{i+1})$ holds for $1 \leq i  < q_{a}$,
 \label{OperadfcOp1} 
 \item for $u$ the root of $S$ we have $dL(u)=p $ and $dR(u)= r$.\label{OperadfcOp2}
%\item $q = \sum_{i=1}^{n} q_{i} -1 +1$. \label{OperadfcOp3} 
\end{enumerate}
\end{mydef}
%\begin{opm}
%Note that an admissible stacking $S$ is also full with respect to $q_1,\ldots,q_n;q$ due to condition \ref{OperadfcOp2}.
%\end{opm}

 Being admissible is well-defined due to the following lemma.
\begin{lemma}
For an admissible stacking $S$ the collection of start values $(s_{a})_{a \text{ leaf of } S}$ is uniquely determined.
\end{lemma}
\begin{proof}
Let $a_{1} \tre\ldots \tre a_{k}$ be the leaves of $S$. Consider $a_{i}$ and $a_{i+1}$: let $b$ be the vertex under both $a_i$ and $a_{i+1}$ maximal for the relation $<$. Then there exist two unique paths $(b,c_{1},\ldots,c_{m},a_{i})$ and $(b,d_{1},\ldots,d_{l},a_{i+1})$.
$$\scalebox{0.8}{$\input{Startvalues_proof.tikz}$}$$
 Consequently, $c_{1}\tre d_{1}$ and they are minimal for this relationship, thus $dR(c_{1}) = dL(d_{1})$, or equivalently 
$$\sum_{i=1}^{m}r_{c_{i}} + r_{a_i} + s_{a_i} = \sum_{i=l}^{m}p_{d_{i}} + p_{a_{i+1}}+ s_{a_{i+1}}$$
As a result, $s_{a_1},\ldots,s_{a_k}$ determine each other, and $s_1$ and $s_{a_k}$ are determined respectively by $p$ and $r$.
\end{proof}
%
%It is clear that example \ref{exDepth} is not admissible. Let us give an admissible example.
%
%\begin{vb}\label{exAdmissible}
%We adapt the stacking of example \ref{exDepth} by adding some boxes and changing the start values
%$$\tikzfig{exAdmissible}$$
%which defines an admissible stacking for the data
%\begin{alignat*}{4}
% &(p_1,q_1,r_1) = (2,3,1) \quad &&(p_4,q_4,r_4) = (1,2,1) \quad &&(p_7,q_7,r_7) = (1,1,1) \quad &&(p_10,q_10,r_10) = (1,1,1) \\
% &(p_2,q_2,r_2) = (0,2,0)\quad &&(p_5,q_5,r_5) = (3,1,1)\quad &&(p_8,q_8,r_8) = (1,1,1) \quad && (p_{11},q_{11},r_{11}) = (1,1,1)\\
% &(p_3,q_3,r_3) = (2,0,3)\quad &&(p_6,q_6,r_6) = (3,1,0)\quad &&(p_9,q_9,r_9) = (2,1,2) &&
% \end{alignat*}
% with unique set of start values
% $$s_3 = 2, \quad s_10 = s_11 = s_9 =s_6 = 0$$
%\end{vb}
%Similarly, the generators $C_{\nth{q}}^{p_0,\ldots,p_n;p,r}$ and $\eta_h^{p,r}$ can be identified with the admissible stackings drawn in \ref{drawings_generators}. Note that the unit is not yet present.

\begin{vb}\label{exstacking}
The following two stackings of boxes are respectively admissible and non-admissible:
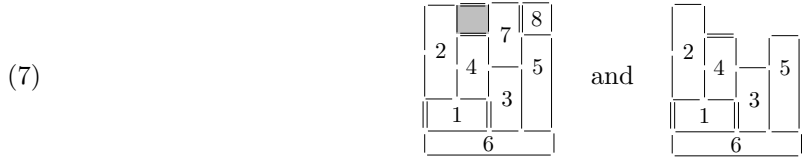
\begin{equation}
\scalebox{0.85}{$\input{stacking_example.tikz}$} \quad \text{ and } \quad \scalebox{0.85}{$\input{stacking_nonexample.tikz}$}
\end{equation}
Note that the grey area represents the start value $s_4=1$. If a label $p,q$ or $r$ is zero, it represents a degenerate side which we draw using two parallel lines representing an equality sign.
\end{vb}

\begin{lemma}
The substitution of stackings is well-defined. Moreover, it preserves admissibility.
\end{lemma}
\begin{proof}
This is well-defined as it is in fact a particular case of substitution of trees. Geometrically it is clear that the substitution of an admissible stacking in an admissible stacking is admissible. 
\end{proof}
The above lemma warrants the following definition.

\begin{mydef}\label{OperadfcOp}
Let $\Stack$ be the $\N^3$-coloured operad of admissible stackings, that is,
\begin{itemize}
\item  $\Stack\left( (p_{1},q_{1},r_1),\ldots,(p_{n},q_{n},r_n);(p,q,r)\right)$ is the set of admissible stackings of the corresponding boxes,
\item the $\Ss$-action is given by permutation of vertices, that is, $S^{\si}$ is the stacking given by renaming box $i$ in $S$ by $\si\inv(i)$,
\item composition is given by substitution of stackings in boxes. 
\end{itemize} 
\end{mydef}

\begin{vb}\label{excomposingstacking}
We compose the admissible stacking from \ref{exstacking} at box $4$ with another admissible stacking 
$$ \scalebox{0.85}{$\input{stacking_example.tikz}$} \quad \circ_4 \quad \scalebox{0.85}{$\input{stacking_example_2.tikz}$} \quad = \quad \scalebox{0.85}{$\input{stacking_composed.tikz}$}$$
\end{vb}

From here on, we will solely work with admissible stackings and thus we simply speak of ``stackings''.

\subsubsection{Adding a unit}

\begin{mydef}
Let $\Stack[\eta]$ be the operad generated by $\Stack$ and a formal element $\eta\in \Stack[\eta](;(0,1,0))$ satisfying the following two relations:
\begin{enumerate}
\item let $i$ be an internal vertex of a stacking $T$ with exactly one child $j$ and let $T'$ be the stacking obtained from $T'$ by contracting the substacking spanned by vertices $i$ and $j$, then we have $T \circ_i \eta = T'$.
\item let $j_1 \tre \ldots \tre j_{q_i}$ be the set of children of an internal vertex $i$ in a stacking $T$ with corresponding number of inputs $q_{j_s}$. Assume for simplicity $j_1 < \ldots < j_{q_i}$ as numbers. Let $T'$ be the stacking obtained from $T$ by contracting the substacking spanned by the vertices $i,j_1,\ldots,j_{q_i}$, then we have 
$T \circ_{j_{q_i}} \eta \circ_{j_{q_i-1}} \ldots \circ_{j_1} \eta =T'$. 
\end{enumerate}
\end{mydef}
\begin{opm}
Remark that these conditions correspond to those in $\S \ref{subsubpartrees}$.
\end{opm}

\subsubsection{Comparing $\Stack[u]$ and $\boxop$}

The generator $C^{p_0,\ldots,p_n;p,r}_{q_1,\ldots,q_n}$ can be uniquely represented by the stacking
\begin{equation}
\scalebox{0.8}{$\input{genFcOp.tikz}$} . \label{drawingGeneratorsBoxop}
\end{equation}
Note in particular that the generator $C^{p_0;p,r}$ corresponds to the stacking
\begin{equation}
\label{drawingHorizontalUnit}
\scalebox{0.8}{$\input{stackingHorizontalUnit.tikz}$}
\end{equation}
which consists of a single box and start number $p_0$.

The associativity condition can then by expressed by the equality of stackings
\begin{equation}
\label{drawingAssocStackings}
\scalebox{0.8}{$\input{Stackings_associative_left.tikz}$} \quad = \quad  \scalebox{0.8}{$\input{Stackings_associative_right2.tikz}$}
\end{equation}

\begin{prop}
We have a morphism of coloured symmetric operads 
$$\psi:\boxop \longrightarrow \Stack[\eta]$$
defined by sending the generator $C^{p_0,\ldots,p_n;p,r}_{q_1,\ldots,q_n}$ to the stackings \eqref{drawingGeneratorsBoxop} and sending the unit $\eta$ of $\boxop$ to the formal element $\eta$.
\end{prop}
\begin{proof}
Picture \eqref{drawingAssocStackings} verifies the associativity. The relations of the unit are readily verified.
\end{proof}

A stacking $S$ is in \emph{standard order} if the underlying tree is in standard order. In the rest of this section, for simplicity, we often omit the indices and write $C_n$ for the generators of $\boxop$.

\begin{lemma}\label{decompStandard}
Let $S \in \Stack$ be a stacking, then either
\begin{itemize}
\item $S$ consists of a single vertex and equals either the operadic unit $1$ or the generator $C_0$,
\end{itemize}
or
\begin{itemize}
\item there exists a unique permutation $\si$ and unique stackings $S_1,\ldots,S_k$ in standard order such that
$$S = \left( \quad \scalebox{0.8}{$\input{decompStacking.tikz}$}\quad \right)^\si  = \left(C_k \circ (- , S_1,\ldots,S_k)\right)^\si$$
\end{itemize} 
\end{lemma}
\begin{proof}
For $S$ consisting of a single vertex, this is trivial. 

For $S$ a stacking with at least two vertices, it is clear that there exists a unique permutation $\si$ such that $S^\si$ is in standard order. Assume $S$ is in standard order, then the result is immediate from the picture. Note that $S_1,\ldots, S_k$ are in standard order since $S$ is.
\end{proof}

\begin{constr}\label{constrMorphStackings}
Let the map $\varphi: \Stack \longrightarrow \boxop$ on a stacking $S$ be defined by induction on the number $n$ of boxes as follows
\begin{itemize}
\item[\underline{$n= 1$}] $S$ is either the operadic unit $1$ or the generator $C_0$, for which we define $\varphi(S)$ as their counterpart in $\boxop$,
\item[\underline{$n >1$}] Given the unique decomposition of $S$ from lemma \ref{decompStandard}, we define 
$$\varphi(S) :=  (C_k \circ_{k+1} \varphi(S_k) \circ_{k} \ldots \circ_2 \varphi(S_1))^\si$$
\end{itemize}
\end{constr}

\begin{prop}
Construction \ref{constrMorphStackings} defines a morphism of coloured symmetric operads and extends to a morphism 
$$\varphi: \Stack[\eta] \longrightarrow \boxop$$
by sending $\eta$ to $\eta$.
\end{prop}
\begin{proof}
By construction, $\varphi$ is equivariant. Consider $S$ and $S'$ standard stackings for which we show that $\varphi(S \circ_i S') = \varphi(S) \circ_i \varphi(S')$. Let $n$ and $m$ be respectively the number of boxes in $S$ and $S'$. We continue by induction on $n+m\geq 2$. 

\proofpart{1}{$n+m = 2$}
In this case $n=m=1$, and thus $S$ and $S'$ are either the operadic identity or the generator $C_0$. In all these cases, the statement is readily verified.

\proofpart{2}{$n+m>2$}
We fix $k = n+m$. We go through the cases $n=1,\ldots,k-1$.

\begin{itemize}
\item[\underline{$n=1$}] $S$ is either the operadic unit $1$ or the generator $C_0$. The first case is trivially satisfied, so we assume the second. The stacking $S'$ thus has a decomposition as $S'= C_{k'} \circ (-, S_1',\ldots,S'_{k'})$. The stacking $C_0 \circ S'$ then differs from $S'$ solely by adding a fixed number $p_0$ to all starting numbers. Hence, we can equivalently write it as adding $p_0$ to the suitable starting numbers of the stackings $S_1',\ldots,S_{k'}'$ 
$$C_0 \circ_1 S' = C_{k'} \circ (-,C_0 \circ_1 S_1',\ldots,C_0 \circ_1 S'_{k'})$$
and thus
$$\varphi(C_0 \circ_1 S') = C_{k'} \circ  (-,\varphi(C_0 \circ_1 S_1'),\ldots,\varphi(C_0 \circ_1 S'_{k'}))$$
Using induction and reordering of the operadic composition, we further compute
$$ C_{k'} \circ (-,C_0,\ldots,C_0) \circ (-,\varphi(S_1'),\ldots,\varphi(S'_{k'}))$$
Using relation \eqref{OperadfcOp1} from $\boxop$, we obtain
$$C_0 \circ_1 (C_{k'} \circ (-,\varphi(S'_1),\ldots,\varphi(S'_{k'}))= C_0 \circ_1 \varphi(S')$$
\item[\underline{$n>1$}] For $n>1$, $S$ has a decomposition as $C_k \circ (-,S_1,\ldots,S_k)$. We discern two cases.

\proofcase{1}{$i>1$}
In this case, we have
$$S \circ_i S' = C_k \circ (-,S_1,\ldots,(S_j \circ_{i'} S')^\tau,\ldots, S_k)$$
for a suitable $j$ and $i$, and a permutation $\tau$ such that $(S_j\circ_{i'}S')$ is in standard order. As $S_j \circ_{i'} S'$ consists of strictly less boxes as $S \circ_i S'$, we can compute 
$$\varphi(S \circ_i S') = C_k \circ (-,\varphi(S_1),\ldots,\varphi((S_j\circ_{i'}S')^\tau), \ldots,\varphi(S_k)) = C_k\circ(-,\varphi(S_1),\ldots,\varphi(S_k)) \circ_i \varphi(S') = \varphi(S) \circ_i \varphi(S')$$

\proofcase{2}{$i=1$}
If $S'$ consists of a single box, $S'= 1$ since $n>1$. Hence, $S \circ_1 S' = S$. Assume now that $S'$ consists of at least two boxes, then consider its decomposition $S'= C_{k'} \circ (-, S_1',\ldots,S'_{k'})$ as in lemma \ref{decompStandard}.

We compute 
$$ \scalebox{0.8}{$\input{decompstacking.tikz}$} \quad \circ_1 \quad \scalebox{0.8}{$\input{StackedBoxes_1.tikz}$} \quad = \quad \scalebox{0.8}{$\input{StackedBoxes_2.tikz}$}$$
for unique permutations $\si_1,\ldots,\si_{k'}$ such that each stacking $S''_i$ is of the form
$$\scalebox{0.8}{$\input{StackedBoxes_3.tikz}$} \quad = C_{l_1} \circ (S_i, S'_{j_i},\ldots,S'_{j'_i}) $$
for certain $j_i, j'_i$. Since each $S''_i$ consists of strictly fewer boxes as $S \circ_i S'$, by induction $\varphi$ acts as a morphism of operads on $S''_i$. We thus compute $\varphi(S\circ_i S')$ as
\begin{align*}
&=C_{k'} \circ (-, \varphi(S^{''}_1),\ldots,\varphi(S^{''}_{k'})) \\
&=(C_{k'} \circ (-,C_{l_1},\ldots,C_{l_{k'}}) \circ (\varphi(S_1'),\varphi(S_1),\ldots,\varphi(S_{l_1}),\ldots,\varphi(S'_{k'}),\varphi(S_{\sum_{j<k'}l_j +1}), \ldots,\varphi(S_k)))^{\tau} \\
&= (C_{k} \circ_1 C_{k'} \circ (-,\varphi(S_1'),\ldots,\varphi(S'_{k'}), \varphi(S_1),\ldots,\varphi(S_k)))^{\tau'} \\
&= C_k \circ (-,\varphi(S_1),\ldots,\varphi(S_k)) \circ_1 (C_{k'} \circ (-,\varphi(S'_1),\ldots,\varphi(S'_{k'})) = \varphi(S) \circ_1 \varphi(S')
\end{align*}
for an appropriate permutation $\tau$. Note that we used relation \eqref{OperadfcOp1} from $\boxop$ to obtain the third equality. 
\end{itemize}
It is easy to verify that $\varphi$ extends to include the unit $\eta$.
\end{proof}

\begin{theorem}\label{ThmBoxopStackings}
The coloured symmetric operads $\boxop$ and $\Stack[\eta]$ are isomorphic.
\end{theorem}
\begin{proof}
It suffices to show that the morphisms $\psi$ and $\varphi$ are inverse to each other. 

The identity $\varphi \psi = \Id_{\boxop}$ is readily verified on generators. Similarly, the identity $\psi \varphi = \Id_{\Stack[\eta]}$ is readily verified following the inductive construction \ref{constrMorphStackings}.
\end{proof}

\subsubsection{The vertical composite map}

\begin{prop}
The morphism
$$\V: \boxop \longrightarrow \Op$$
from Proposition \ref{morphismtrees} sends a stacking $S$ to its underlying tree $\V(S)=S$.
\end{prop}
\begin{proof}
Sending a stacking to its underlying tree is clearly a morphism of operads as both are defined using tree substitution. As they coincide on the generators, they are equal as morphisms.
\end{proof}

\subsubsection{The horizontal composite graph}

\begin{constr}\label{constrGraph}
Let $S$ be a stacking with $n$ boxes, then we construct its \emph{horizontal composite graph} $\Hor_S$ inductively as follows
\begin{itemize}
%\item[\underline{$n=0$}] The unit $\eta$ is send to the unit $\eta_P$ 
%$$  \scalebox{0.7}{$\tikzfig{imageEta}$}$$
\item[\underline{$n= 1$}] $S$ is either the operadic unit $1$ or the generator $C^{p_0;p,r}$, for which we define
$$\scalebox{0.7}{$\input{imageHorizontalC_0.tikz}$} $$
\item[\underline{$n >1$}] Given the unique decomposition of $S$ from lemma \ref{decompStandard}, then the horizontal composite graphs $\Hor_{S_1},\ldots,\Hor_{S_k}$ are well-defined and we define $\Hor_S$ as the following graph
$$  \scalebox{0.7}{$\input{horizontalGraphC_0.tikz}$}$$
\end{itemize}
\end{constr}
%
%\begin{align*}
%\scalebox{0.7}{$\tikzfig{unit}$} \quad &\longmapsto \quad \scalebox{0.7}{$\tikzfig{imageEta}$} \\
%\scalebox{0.7}{$\tikzfig{genFcOp_unit}$} \quad &\longmapsto \quad \scalebox{0.7}{$\tikzfig{imageHorizontalEta}$}\\
%\scalebox{0.7}{$\tikzfig{genFcOp}$} \quad &\longmapsto \quad \scalebox{0.7}{$\tikzfig{imageCorolla}$}
%\end{align*}

\begin{prop}
The morphism
$$\Hor: \boxop \longrightarrow \Pro$$
from Proposition \ref{morphismgraphs} sends a stacking $S$ to the planar graph from construction \ref{constrGraph} and $\eta$ to $\eta_P$.
\end{prop}
\begin{proof}
Construction \ref{constrGraph} is the translation of the composition $\Stack \longrightarrow \boxop \overset{\Hor}{\longrightarrow} \Pro$.
\end{proof}

\subsubsection{$\boxop$ as suboperad of $\Op \times \Pro$}

\begin{theorem}\label{thmboxopSUBoppro}
The morphism
$$ (\V,\Hor): \boxop \longrightarrow \Op \times \Pro $$
 is injective.
\end{theorem}
\begin{proof}
Consider $(T,G) = (\V(S), \Hor(S)) \in \Op \times \Pro$. First, observe that $T = \V(S)$ is the underlying tree of the stacking $S$, which is thus uniquely determined. The question remains whether the boxes $(p_i, q_i, r_i)$ which are part of the datum of the stacking $S$ are also uniquely determent. Clearly, $q_i$ is determined by $\V(S)$ whereas $p_i$ and $r_i$ are determined by $\Hor(S)$.
%For every stacking $S$ its underlying tree $\V_S = S$ is unique.
\end{proof}
\begin{opm}
Note that the morphism $\V$ is not injective.
\end{opm}

We provide an illustrative example.

\begin{vb}\label{exTreeGraph}
The (admissible) stacking from example \ref{exstacking} has the following vertical composite tree and horizontal composite graph
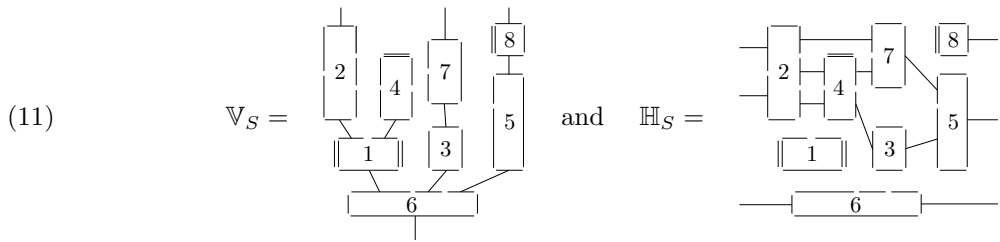
\begin{equation}\label{verticalhorizontal}
 \V_S = \quad \scalebox{0.85}{$\input{vertical_composite_tree_example.tikz}$} \quad \text{ and } \quad \Hh_S = \quad \scalebox{0.85}{$\input{horizontal_composite_graph_example.tikz}$} 
 \end{equation}
They are part of the coloured symmetric operads $\boxop,\Op$ and $\Pro$ respectively as follows
\begin{itemize}
\item $S \in \boxop((0,2,0),(2,1,3),(1,1,1),(2,0,2),(2,1,1),(1,3,1),(2,1,1),(0,1,1);(3,3,3))$,
\item $\V_S \in \Op(2,1,1,0,1,3,1,1;3)$,
\item $\Hh_S \in \Pro((0,0),(2,3),(1,1),(2,2),(2,1),(1,1),(2,1),(0,1);(3,3))$.
\end{itemize}
\end{vb}

\section{Higher Gerstenhaber brackets} \label{parbrackets}

In the seminal work \cite{gerstenhabervoronov}, Gerstenhaber and Voronov totalise a linear nonsymmetric operad $\ppp$ into a graded $k$-module $\Tot(\ppp)$, placing elements of arity $q$ in degree $q$. Moreover, they show the desuspension $s^{-1}\Tot(\ppp) = \Tot(s^{-1}\ppp)$ to be a brace algebra, including the first Gerstenhaber brace which induces the Gerstenhaber Lie bracket. The authors go on to show that an operad with multiplication carries a homotopy $G$-algebra structure.

The main goal of this section is to develop a similar approach to the higher structure of a linear box operad $\bbb$, that is, a $k$-algebra over the linearised $k\boxop$. However, a problem arises which can be seen directly on the level of symmetric operad $k \boxop$. Contemplating the possibility of having a Gerstenhaber brace between boxes, it becomes clear that in general stacking a box $(p,q,r)$ on top of another box only works in case
\begin{equation}\label{thin}
(p,q,r) = (0,q,0)
\end{equation}
which ensures that the result will again be a box. Boxes satisfying \eqref{thin} are called \emph{thin}, and the thin boxes will play a crucial role in what follows.

After having described the morphism $\mathsf{PreLie} \rightarrow \Op_{\mathrm{d}}$ corresponding to the case of nonsymmetric operads $\ppp$ in \S \ref{subparbracket}, we start by establishing the appropriate totalisation for box operads $\bbb$ in \S \ref{subpartot}. As in the case of nonsymmetric operads, we have to desuspend, but this time we only have to do so in as far as thin boxes are concerned. We thus obtain the totalised symmetric graded operad $\boxop_{\mathrm{td}}$ where the subscript stands for ``thin desuspended''.

Next, in \S \ref{subparlinf} we prove our main result, which is the existence of a morphism of symmetric dg operads
\begin{equation}
        \Linf \rightarrow \boxop_{\td}: l_n \mapsto L_n
\end{equation}
where $\boxop_{\td}$ carries the zero differential (Theorem \ref{Linf_theorem}). The $L_n$ are the anti-symmetrisations of elements $P_n$, where $P_n$ is the summation of all stackings of degree $2-n$ (satisfying some additional technical conditions), with appropriate signs (which are relegated to the Appendix \ref{appsign}).
 
 \subsection{The Gerstenhaber bracket}\label{subparbracket}

Following \cite{gerstenhabervoronov}, for $\ppp$ a linear operad, its desuspended totalisation $\Tot(s\inv\ppp)$ is defined in degree $n$ as 
$$s\inv\Tot(\ppp)^n = \Tot(s\inv\ppp)^n = \ppp(n+1)$$
It carries a $\PreLie$-structure: for $\te \in \ppp(k)$ and $\te'\in \ppp(l)$, the first Gerstenhaber brace is defined as
$$\te \bullet \te'  := \sum_{i=1}^{k} (-1)^{(i-1)(l-1)} \te \circ_i \te'$$
Note it has degree $0$. Its anti-symmetrization constitutes a Lie bracket
$$[\te,\te'] := \te \bullet \te' - (-1)^{|\te||\te'|} \te' \bullet \te $$
called the \emph{Gerstenhaber bracket}.

\subsubsection{The morphism $\PreLie \longrightarrow \Op_d$}
We translate this to the operadic level. In order to define a morphism $\PreLie \longrightarrow \Op_{\dd}$, we first need to compile the linearised $\N$-coloured symmetric operad $k\Op$ into a graded (uncoloured) symmetric operad $\Op_{\dd}$ shifted by the function $\dd(n):= n-1$ (see Appendix \ref{subapptotop}). Explicitly, this is defined as follows: let
$$\Op_{\dd}(n) \subseteq \prod_{( \nth{q};q)\in \N^{n+1}} k\Op(\nth{q};q)$$
be the subspace generated as $k$-module by sequences of constant degree, where an element $x \in \Op(\nth{q};q)$ is placed in degree $q-1 - \sum_{i=1}^n(q_i-1)$. We write $x^{\dd} \in \Op_{\dd}$ and denote its degree by $|x^{\dd}|$. In this case, $\Op_{\dd}$ is concentrated in degree $0$. The composition of $\Op_{\dd}$ is derived from the composition of $\Op$ with appropriate signs: for $x \in \Op(\nth{q};q)$ and $x' \in \Op(\fromto{q'}{m};q_i)$, we set
$$x^{\dd} \circ_i x^{'\dd} := (x \circ_i x')^{\dd}$$
and the $\Ss$-action is given as $(x^{\dd})^\si = (-1)^{ \si(q-1)} (x^\si)^{\dd}$, where the sign is the koszul sign for permuting elements of degree $q_i-1$.

The Gerstenhaber brace is defined using partial compositions, that is, the elements
$$E_i^{k,l} := C_{1,\ldots,l,\ldots,1} \circ (\underbrace{\eta,\ldots,\eta}_{(i-1)\text{-times}}, - ,\underbrace{\eta,\ldots,\eta}_{(k-i)\text{-times}}) \in \Op(k,l;k+l-1).$$

\begin{prop}\label{gerstenhaber}
We have a morphism of operads
$$\PreLie \longrightarrow \Op_{\dd}: p_2 \longmapsto P_2$$
defined as 
$$(P_2)_{k,l} := \sum_{i=1}^n (-1)^{(k-i)(l-1)} (E^{k,l}_i)^{\dd}$$ 
\end{prop}
%\begin{proof}
%We compute that 
%$$(P_2 \circ_1 P_2 - P_2 \circ_2 P_2) - (P_2 \circ_1 P_2 - P_2 \circ_2 P_2)^{(23)} = 0$$
%First, for a given $k,l,m$ we compute that 
%\begin{align*}
%(P_2 \circ_1 P_2)_{k,l,m;k+l+m-2} &= \sum_{s=1}^{k+l-1} \sum_{t=1}^k (-1)^{(k+l-1-s)(m-1) + (k-t)(l-1)} \dd(E_s^{k+l-1,m}) \circ_1 \dd(E_t^{k,l}) \\
%(P_2 \circ_2 P_2)_{k,l,m;k+l+m-2} &= \sum_{s=1}^k \sum_{t=1}^l (-1)^{(k-s)(l+m) + (l-t)(m-1)} \dd(E^{k,l+m-1}_s) \circ_2 \dd(E^{l,m}_t)
%\end{align*}
%On the level of $\Op$ we have the equalities 
%$$E_s^{k+l-1,m} \circ_1 E_t^{k,l} = (E_t^{k+m-1,l} \circ_1 E^{k,m}_{s-l+1} )^{(23)}$$
%for $1 \leq t \leq k$ and $t+l \leq s \leq k+l-1$, and
%$$E_s^{k+l-1,m} \circ_1 E^{k,l}_t = E^{k,l+m-1}_t \circ_2 E^{l,m}_{s-t+1}$$
%for 	$1 \leq t \leq k$ and $t \leq s < t+l$.
%
%As a result, they every term appears exactly twice. It is an easy computation that their corresponding signs differ by exactly $-1$. Note that $\circ_i$ does not add a sign in this case, except the permutation $(23)$.
%%
%% We compute the signs on the left and right hand side
%%\begin{center}
%%\begin{tabular}{c|c || c}
%%left & right & total \\
%%\hline
%%$(k+l-1-s)(m-1) + (k-t)(l-1)$ & $(k+m-1-t)(l-1) + (k-s+l-1)(m-1)+(l-1)(m-1)$ & $0$ \\
%%$(k+l-1-s)(m-1) + (k-t)(l-1)$ & $(k-t)(l+m) + (l-s+t-1)(m-1)$ & $0$
%%\end{tabular}
%%\end{center}
%%Hence, the first equations is part of the term $P_2\circ_1 P_2 - (P_2 \circ_1 P_2)^{(23)}$ and the second of $P_2\circ_1 P_2 - P_2 \circ_2 P_2$.
%\end{proof}

Now given a linear operad $\ppp$, we have an induced $\Op_{\dd}$-structure on the desuspended totalization $\prod_{\dd}\ppp = \Tot(s\inv \ppp)$ (see Corollary \ref{algebralift}). 
%Explicitly, the morphism of coloured symmetric operads $\Op \overset{\ppp}{\longrightarrow }\End(\ppp)$ lifts to morphism of graded symmetric operads 
%$$\ppp^{\dd}: \Op_{\dd} \longrightarrow \End(\Tot(s\inv \ppp))$$
%which is defined for $T \in \Op(\nth{q};q)$ and $x_i \in \ppp(q_i)$ as 
%$$\ppp^{\dd}(T^{\dd})(x_1^{\dd},\ldots,x_n^{\dd})= (-1)^{ \sum_{i=1}^n (q_i-1) \sum_{j>i}(q_j-1)} \ppp(T)(x_1,\ldots,x_n)^{\dd}$$
%Thus, we obtain for $\te\in \ppp(k)$ and $\te'\in \ppp(l)$ that
%$$P_2^{\dd}(\te^{\dd},\te^{'\dd}) = \sum_{i=1}^k (-1)^{(i-1)(l-1)} (\te \circ_i \te')^{\dd}$$
%coinciding with the Gerstenhaber brace defined above.

\begin{prop}\label{propGerstenhaberbracket}
For a linear operad $\ppp$, we have morphisms of operads 
$$ % https://tikzcd.yichuanshen.de/#N4Igdg9gJgpgziAXAbVABwnAlgFyxMJZARgBoAGAXVJADcBDAGwFcYkQAdDgeTQH1gXKFAC+IEaXSZc+QinIVqdJq3ZcACgCcYAGSxsJU7HgJEATIpoMWbRJw4BRMFAAUcLljC0ABFzT+ASnElGCgAc3giUAAzTQgAWyQyEBwIJHJDEFiE9JpUpDMRShEgA
\begin{tikzcd}
\PreLie \arrow[r] & \Op_{\dd} \arrow[r] & \End(\Tot(s\inv \ppp))
\end{tikzcd} $$
where for $\te\in \ppp(k)$ and $\te'\in \ppp(l)$ we have
$$P_2^{\dd}(\te^{\dd},\te^{'\dd}) = \sum_{i=1}^k (-1)^{(i-1)(l-1)} (\te \circ_i \te')^{\dd}$$
\end{prop}

\subsection{The complex $\Tot(s_{\thin}\inv \bbb)$}\label{subpartot}

A natural extension of the Gerstenhaber brace for box operads would require an analogue of the partial composition. Unfortunately, for two boxes $\te \in \bbb(p,q,r)$ and $\te' \in \bbb(p',q',r')$, the tree stacking $\te'$ at the $i$-th input of $\te$ is not part of $\boxop$
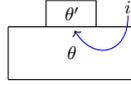
$$ \scalebox{0.7}{$\input{partial_comp_problem.tikz}$} $$
However, if $(p,r)=(0,0)$, the stacking 
$$ \scalebox{0.7}{$\input{partial_comp_new.tikz}$} $$
makes sense. We employ this observation to remedy the situation.

\subsubsection{Thin boxes}\label{thinrectangles}

A box of size $(p,q,r)$ is \emph{thin} if $p=r=0$. We draw a thin box as a box with degenerate sides
$$\scalebox{0.7}{$\input{GSrectangle_flat.tikz}$}$$
Note in particular that for a thin box $\te'$, we have a partial composition $\te\circ_i \te'$ defined as
$$\scalebox{0.8}{$\input{partial_comp_new.tikz}$} \quad := \quad \scalebox{0.8}{$\input{partial_comp_units_new.tikz}$}$$
\begin{opm}
Note that the subcollection of thin boxes $(\bbb(0,q,0))_{q\geq 0}$ of a box operad forms an operad. Reversely, an operad can be considered as a box operad concentrated in thin sizes.
\end{opm}

A box of size $(p,q,r)$ is \emph{semi-thin} if $(p,r)\neq (0,0)$ and $p=0$ or $r=0$, that is, boxes of the form
$$ \scalebox{0.8}{$\input{leftsemithin.tikz}$} \quad \text{ or } \quad \scalebox{0.8}{$\input{rightsemithin.tikz}$}$$
A box is \emph{non-thin} if it is neither thin, nor semi-thin.

\subsubsection{The complex $\Tot(s_{\thin}\inv\bbb)$}

Let $M$ be a collection $(M(p,q,r))_{p,q,r\in \N}$ of $k$-modules, then we associate a totalized graded $k$-module $\Tot(M)$ which is defined in degree $n$ as
 $$\Tot(M)_n := \prod_{ (p-r) + q =n } M(p,q,r)$$
Similar to the Hochschild object, we aim to desuspend. However, for box operads we do so solely for thin boxes.

The \emph{thin desuspension} of $M$ is then defined as 
 $$s_{\thin}\inv M(p,q,r)_{k} := \begin{cases}  M(0,q,0)_{k+1} & \text{ if } (p,r)=(0,0) \\ M(p,q,r)_k &\text{ otherwise } \end{cases}$$
Note that the desuspended Hochschild object is contained in $\Tot(s_{\thin}\inv\bbb)$, yet the graded $k$-modules $s\inv \Tot(\bbb)$ and $\Tot(s_{\thin}\inv \bbb)$ no longer coincide.

\subsubsection{The operad $\boxop_{\td}$}

First, we define the function 
$$\td: \N^3 \longrightarrow \Z: \td(p,q,r) := \begin{cases} p-r+q & \text{ if } (p,r) \neq (0,0) \\
q-1 & \text{ otherwise } \end{cases}$$
which stands for \emph{thin desuspended}.

Next, we compile the linearised $\N^3$-coloured operad $k \boxop$ and obtain a graded (uncoloured) operad $\boxop_{\td}$ shifted by $\td$ (see Appendix \ref{subapptotop}). Explicitly, it is defined as 
$$\boxop_{\td}(n) \sub \prod_{\substack{\nth{q},q \\ \nth{p},p \\ \nth{r},r }} k\boxop\left( (p_1,q_1,r_1),\ldots,(p_n,q_n,r_n); (p,q,r) \right)$$
the subspace generated as $k$-module by sequences of constant degree, where a stacking 
$$S \in \boxop((p_1,q_1,r_1),\ldots,(p_n,q_n,r_n);(p,q,r))$$
 is placed in degree
  $$|S^{\td}|:= \td(p,q,r) - \sum_{i=1}^n\td(p_i,q_i,r_i).$$
  We call this the \emph{thin desuspended degree}. We write $S^{\td} \in \boxop_{\td}$. The composition of $\boxop_{\td}$ is derived from the composition of $\boxop$ with appropriate signs: for $S' \in \boxop$ composable, we set
$$S^{\td} \circ_i S^{'\td} = (-1)^{ |S^{'\td}| \sum_{j<i}\td(p_j,q_j,r_j)} (S \circ_i S')^{\td}$$
and the $\Ss$-action is given as $(S^{\td})^\si = (-1)^{ \si(\td(p,q,r))} (S^\si)^{\td}$, where the sign is the koszul sign for permuting elements of degree $\td(p_i,q_i,r_i)$.
\begin{lemma}
For a stacking $S \in \boxop$ not containing semi-thin boxes, we have
\begin{equation} \label{degreeRectangle}
 |S^{\td}| = \begin{cases} (1+\text{\# number of thin boxes}) - n &  \text{ if } (p,r) \neq (0,0) \\ 0 & \text{ if }  (p,r)=(0,0) \end{cases}
 \end{equation}
\end{lemma}
\begin{proof}
Assume $S$ does not contain semi-thin boxes, then the resulting box of $S$ is thin if and only if $S$ consists solely of thin boxes. Hence, the formula is a direct consequence of $q-1= \sum_{i=1}^n (q_i-1)$ and $p-r= \sum_{i=1}^n (p_i-r_i)$ for a stacking $S \in \boxop((p_1,q_1,r_1),\ldots,(p_n,q_n,r_n;(p,q,r))$.
\end{proof}
 A consequence of the sign choice is that for $\bbb$ a box operad, we have an induced $\boxop_{\td}$-algebra structure on $\Tot(s_{\td}\inv \bbb)$ (Corollary \ref{algebralift}). Note that, for $\te_i \in \bbb(p_i,q_i,r_i)$, we have 
\begin{equation}
 S^{\td}(\te_1^{\td},\ldots,\te_n^{\td}) := (-1)^{ \frac{|S^{\td}|(|S^{\td}|+1)}{2} + |S^{\td}|\td(p,q,r) + \sum_{i=1}^n |\te_i^{\td}| \sum_{j>i}\td(p_j,q_j,r_j)  } S(\te_1,\ldots,\te_n)^{\td} \label{signsStacking}
\end{equation}

\subsection{The morphism $\Linf \longrightarrow \boxop_{\td}$}\label{subparlinf}

For a box operad, we need a replacement of the partial compositions $E^{k,l}_i \in \Op$. As these trees compose two elements into one, we call them quadratic. Similarly for box operads, we focus on thin boxes and consider \emph{thin-quadratic} stackings. We use these in order to define an appropriate $\Linf$-structure on $\Tot(s_{\td}\inv \bbb)$ expressed by a morphism $\Linf \longrightarrow \boxop_{\td}$. Moreover, in analogy with the Gerstenhaber bracket, they are anti-symmetrisations of operations $P_n \in \boxop_{\td}$ and thus constitute an $\Linf$-admissible structure (Definition \ref{Linfadm}). 

\subsubsection{$\Linf$-admissible algebras}

Recall that a Lie-admissible algebra satisfies the minimal axioms in order for its anti-symmetrisation to be a Lie algebra. Similarly, an $\Linf$-admissible algebra has a sequence of operations satisfying the minimal axioms such that its anti-symmetrisation constitutes an $\Linf$-algebra.

% Concretely, a $\Linf$-admissible algebra consists of a dg-module $(V,d)$ and a chain map $p_n: V^{\otimes n} \longrightarrow V$ of degree $2-n$ for every $n\geq 2$ such that their antisymmetrization
%$$l_n := \sum_{\si \in \Ss_n} (-1)^{\si} p_n^\si$$
%defines a $\Linf$-structure on $(V,d)$. 
This is encoded by the following dg-operad.

\begin{mydef}\label{Linfadm}
Let $\Linfadmiss$ be the dg-operad generated by the elements $p_n \in \Linfadmiss(n)$ of degree $2-n$ satisfying the relation
\begin{equation*}
\partial(p_n) = \sum_{\substack{ k+l = n+1 \\ k,l\geq 2}} \sum_{\si \in \Ss_{n}} \sum_{j=1}^{k} (-1)^{(k-1)l + (j-1)(l-1)} (-1)^{\si}\left(P_k \circ_j P_l\right)^{\si} 
 \end{equation*}
for every $n \geq 2$.
\end{mydef} 
Note that the above relation is defined exactly such that 
$$\Linf \longrightarrow \Linfadmiss: l_n \longmapsto \sum_{\si \in \Ss_n} (-1)^{\si} p_n^\si$$
is a morphism of dg-operads.

\subsubsection{Thin-quadratic stackings}

We consider stackings which compose $n$ boxes into a new box such that the number of thin boxes reduces by exactly $1$. We can succinctly express this by saying its degree should be $2-n$ (see \eqref{degreeRectangle}).

In order to obtain an $\Linf$-structure, we add some coherence conditions on the stackings.  
\begin{mydef}
Let $S\in \boxop$ be a stacking of $n$ boxes, then $S$ is \emph{thin-quadratic} if
\begin{enumerate}
\item $S$ has degree $2-n$,
\item $S$ is in standard order,
\item $S$ has a horizontal composite graph $\Hh_S$ with exactly two connected components,
\item $S$ does not contain semi-thin boxes.\label{rightthin}
\end{enumerate}
We write $\boxop^{2-n}_{\mathsf{tq}}(n)$ for the set of thin-quadratic stackings of $n$ boxes.
\end{mydef}
A thin-quadratic stacking $S$ appears in three types:
\begin{itemize}
\item $S$ is \emph{of type $\rom{1}$} if the resulting box is thin. In this case, we obtain the usual partial stackings of $2$ boxes
\begin{equation}\label{type1} \scalebox{0.8}{$\input{P_2_1.tikz}$} \end{equation}
\item  $S$ is \emph{of type $\rom{2}$} if both the resulting box and the bottom box are non-thin. In this case, both connected components consists of a single box, i.e. the partial stackings
\begin{equation}\label{type2} \scalebox{0.7}{$\input{partial_comp.tikz}$}
\end{equation}
\item $S$ is \emph{of type $\rom{3}$} if the resulting box is non-thin and the bottom box is thin. In this case, the stacking possibly has an arbitrary number of boxes and is of the following form
\begin{equation}\label{type3}
 \scalebox{0.8}{$\input{P_n.tikz}$} \end{equation}
where the top layer is horizontally connected.
\end{itemize}
We write respectively 
$\boxop^{2-n}_{\mathsf{tq},\rom{1}}(n)$, $\boxop^{2-n}_{\mathsf{tq},\rom{2}}(n)$, $\boxop^{2-n}_{\mathsf{tq},\rom{3}}(n)$.

%$ \Thin^2_{\rom{1}}, \Thin^2_{\rom{2}}$ and $\Thin^2_{\rom{3}}$.

\subsubsection{Signs}

In order to define the higher Gerstenhaber braces in \S \ref{parhighergerst}, we require a sign $(-1)^{S^{\td}}$ for every stacking $S \in \boxop$. We relegate this technicality to Appendix \ref{appsign} (Definition \ref{signStacking}), in order not to obscure the main result. 

For a thin-quadratic stacking $S \in \boxop((p,q_1,r),(0,q_2,0);(p,q_1+q_2-1,r))$ of type $\rom{1}$ and $\rom{2}$ where the top box is plugged at input $i$ of the bottom box, the corresponding sign is  
\begin{equation}
(-1)^{(q_1-i+p-r)(q_2-1)} \label{signsQuadraticThin}
\end{equation}

For thin-quadratic stackings of type $\rom{3}$, we do not have a closed formula.

\subsubsection{Higher Gerstenhaber braces $P_n$}\label{parhighergerst}

%We construct elements $P_n \in \boxop_{\td}$ and show that their antisymmetrizations $L_n$ form a morphism $\Linf \longrightarrow \boxop_{\td}$. 

\begin{mydef}
For $n\geq 2$, we define the \emph{$n$-Gerstenhaber brace} $P_n \in \boxop_{\td}(n)$ as
$$P_n := \sum_{S \in \Thin^2(n)} (-1)^{S^{\td}} S^{\td} $$
which has degree $2-n$. We define the \emph{$n$-Gerstenhaber bracket} $L_n$ as their anti-symmetrisation, i.e.
$$L_n := \sum_{\si \in \Ss_n} (-1)^\si L_n^\si.$$
\end{mydef}

 \begin{theorem}\label{Linf_theorem}
 We have a morphism of dg-operads 
 $$\Linfadmiss \longrightarrow \boxop_{\td} : p_n \longmapsto P_n$$
 where $\boxop_{\td}$ carries the zero differential.
 \end{theorem}
 \begin{proof}
 For every $n\geq 3$, we need to show that 
%$$ \sum_{\substack{ k+l = n+1 \\ k,l\geq 2}}\sum_{\chi \in Sh_{k-1,l}} (-1)^{(k-1)l} (-1)^{\chi} \left(L_k \circ_1 L_l\right)^{\chi\inv} = 0$$
%which translates to the equation
\begin{equation*}\sum_{\substack{ k+l = n+1 \\ k,l\geq 2}} \sum_{\si \in \Ss_{n}} \sum_{j=1}^{k} (-1)^{(k-1)l + (j-1)(l-1)} (-1)^{\si}\left(P_k \circ_j P_l\right)^{\si} = 0 \label{Peq}
 \end{equation*}	
 Note that we also have $(-1)^{(k-1)l + (j-1)(l-1)}  =(-1)^{(k-j)l + j-1}$.

Now given a composable pair of thin-quadratic stackings
\begin{itemize}
\item  $S \in \boxop((p_1,q_1,r_1),\ldots,(p_k,q_k,r_k);(p,q,r))$,
\item  $S' \in \boxop((p'_1,q'_1,r_1'),\ldots,(p'_l,q'_{l},r'_{l});(p_i,q_i,r_i))$,
\end{itemize}
and an index $1 \leq i \leq n$, we show that there exists two unique thin-quadratic stackings 
\begin{itemize}
\item  $\tilde{S} \in \boxop((\tilde{p}_1,\tilde{q}_1,\tilde{r}_1),\ldots,(\tilde{p}_{\tilde{k}},\tilde{q}_{\tilde{k}},\tilde{r}_{\tilde{k}});(p,q,r))$,
\item  $\tilde{S}' \in \boxop((\tilde{p}'_1,\tilde{q}'_1,\tilde{r}'_1),\ldots,(\tilde{p}'_{\tilde{l}},\tilde{q}'_{\tilde{l}},\tilde{r}'_{\tilde{l}});(\tilde{p}_j,\tilde{q}_j,\tilde{r}_j))$,
\end{itemize}
a unique index $1\leq j \leq \tilde{k}$ and a unique permutation $\si \in \Ss$ such that 
$$S \circ_i S' = (\tilde{S} \circ_j \tilde{S}')^\si.$$

We discern six cases. Given that the resulting box is thin, we have two cases. For $s \geq t + q_2'$, we have
\begin{equation}
\label{Linfeq1}
\scalebox{0.8}{$\input{case1_1.tikz}$} \quad \circ_1 \quad  \scalebox{0.8}{$\input{case1_2.tikz}$} \quad = \quad  \scalebox{0.8}{$\input{case1_3.tikz}$} \quad = \quad ( \quad  \scalebox{0.8}{$\input{case1_4.tikz}$} \quad \circ_1 \quad  \scalebox{0.8}{$\input{case1_5.tikz}$} \quad )^{(23)}
\end{equation}

For $ t \leq s < t + q_2'$ we have
 \begin{equation}
 \label{Linfeq2}
 \scalebox{0.8}{$\input{case1_1.tikz}$} \quad \circ_1 \quad  \scalebox{0.8}{$\input{case1_2.tikz}$} \quad = \quad  \scalebox{0.8}{$\input{case2_1.tikz}$} \quad = \quad  \scalebox{0.8}{$\input{case1_4.tikz}$} \quad \circ_2 \quad  \scalebox{0.8}{$\input{case2_2.tikz}$}
 \end{equation}

Now, assume the resulting box is not thin. If the bottom box of $X$ is non-thin, we have two very similar cases as above: 
\begin{align}
\label{Linfeq3} 
  \scalebox{0.8}{$\input{case3_1.tikz}$} \quad \circ_1 \quad  \scalebox{0.8}{$\input{case3_2.tikz}$} \quad &= \quad  \scalebox{0.8}{$\input{case3_3.tikz}$} \quad = \quad ( \quad \scalebox{0.8}{$\input{case3_4.tikz}$} \quad \circ_1 \quad  \scalebox{0.8}{$\input{case1_5.tikz}$} \quad )^{(23)} \\
 \label{Linfeq4}
   \scalebox{0.8}{$\input{case3_1.tikz}$} \quad \circ_1 \quad  \scalebox{0.8}{$\input{case3_2.tikz}$} \quad &= \quad  \scalebox{0.8}{$\input{case4_1.tikz}$} \quad = \quad  \scalebox{0.8}{$\input{case3_4.tikz}$} \quad \circ_2 \quad  \scalebox{0.8}{$\input{case2_2.tikz}$}
\end{align}
for similar $s$ and $t$.

For the remaining terms, in which the bottom box of $S$ is thin, we discern two cases: whether the box $i$ in $S$ has a child or not (or equivalently, whether $i$ is a leaf or not). 

In the first case, we compute 
\begin{equation}
\label{Linfeq5}
\scalebox{0.8}{$\input{case5_1.tikz}$} \quad \circ_i \quad  \scalebox{0.8}{$\input{case5_2.tikz}$} \quad = \quad  \scalebox{0.8}{$\input{case5_3.tikz}$} \quad = \quad ( \quad \scalebox{0.8}{$\input{case5_4.tikz}$} \quad \circ_1 \quad  \scalebox{0.8}{$\input{case5_5.tikz}$} \quad )^{(i+1 \ldots k+1)}
\end{equation}
Note that we had to apply the permutation in order to obtain stackings in standard form. 

In the last case, the box $i$ has a child. Hence, we have 
\begin{equation}
\label{Linfeq6}
\scalebox{0.8}{$\input{case6_1.tikz}$} \quad \circ_i \quad  \scalebox{0.8}{$\input{case5_2.tikz}$} \quad = \quad  \scalebox{0.8}{$\input{case6_2_fc.tikz}$}
\end{equation}
where $j+1$ corresponds to the left-most child of $i+1$ in the composite $S \circ_i S'$. As there are no thin or semi-thin boxes above $i$, the children of $i$ define a unique horizontally connected component above $i$ (denoted by the dashed inner stacking). This constitutes a thin-quadratic stacking, save for it being in standard order. Similarly, if we contract this substacking to a single box, the result defines a thin-quadratic stacking as well, except being in standard order. Hence, we have a unique permutation $\si$ such that
\begin{align*}
    \scalebox{0.8}{$\input{case6_2_fc.tikz}$} \quad = \quad ( \quad \scalebox{0.8}{$\input{case6_3.tikz}$} \quad \circ_j \quad  \scalebox{0.8}{$\input{case6_4.tikz}$} \quad )^{\si}
    \tag{\ref{Linfeq6}}
\end{align*}
where the right two boxes are thin-quadratic (i.e. also in standard order).

By lemma \ref{signsLinf} their corresponding signs in all six cases
$$(-1)^{(k-i)l + i-1 + S^{\td} + S^{'\td} } \text{ and }  (-1)^{(\tilde{k}-j)\tilde{l} + j-1 + \si + \tilde{S}^{\td} + \tilde{S'}^{\td}} $$ 
are opposite, thus finishing the proof.
 \end{proof}

\section{Deformation theory of Lax prestacks}\label{parGS}

In this section we apply our main Theorem \ref{Linf_theorem} to the deformation theory of lax prestacks. Here, a \emph{lax prestack} is a lax functor $\A: \mathcal{U} \rightarrow \mathsf{Cat}(k)$ from a small category $\mathcal{U}$ to the category of small linear categories $\mathsf{Cat}(k)$ (Definition \ref{deflaxprestack}).

We start by recalling  the basic setup for the deformation theory of algebras in \S \ref{subpardefassoc}, which is entirely determined by the Hochschild object $\CC(A)$ of a $k$-module $A$ and a ``multiplication'' $m \in \CC^2(A)$ satisfying the Maurer-Cartan equation.

In \S \ref{subparGS} we start with the datum of a \emph{quiver} $\A$ on $\mathcal{U}$ (Definition \ref{defQuiver}) to which we associate an $\N^3$-graded $k$-module $\CC_{\square}^{p,q,r}(\A)$ which is a natural enlargement of the $\N^2$-graded Gerstenhaber-Schack object $\CC_{GS}^{p,q}(\A)$ from \cite{DVL}. We show that $\CC_{\square}^{p,q,r}(\A)$ has the structure of a box operad, and we call it the \emph{Gerstenhaber-Schack box operad} of $\A$.
Next, we show that $\Tot (s^{-1}\CC_{GS}(\A))$ is an $\Linf$-subalgebra of $\Tot (s^{-1}_{\thin}\CC_{\square}(\A))$ (Proposition \ref{propGSLinf}).

In \S \ref{subparlax} we show that a Maurer-Cartan element $\alpha = (m,f,c) \in \Tot (s^{-1}\CC_{GS}(\A))$ corresponds precisely to a lax prestack structure $(m,f,c)$ on $\A$.
As a consequence, using the general machinery of twisting \cite{merkulov2000} \cite{dotsenko_shadrin_vallette_2023} (over $k$ a field of characteristic $0$), we obtain that the deformation theory of a (lax) prestack $(\A, \alpha)$ is governed by the $\alpha$-twisted $\Linf$-algebra $\Tot_{\alpha} (s^{-1}\CC_{GS}(\A))$.

\subsection{Deformation theory of associative algebras}
\label{subpardefassoc}
For $V$ a $k$-module, let $\End(V)$ be its endomorphism operad. The associated \emph{Hochschild object} $\CC(V)$ is defined as the totalisation $\Tot(\End(V))$, that is, its degree $n$ component corresponds to $\End(V)(n)= \Hom(V^{\otimes n},V)$. 

The Gerstenhaber brace equips the desuspended Hochschild object $s\inv \CC(V)$ with a $\mathsf{PreLie}$-structure. An element $m \in s\inv \CC(V)= \CC^2(V)$ is \emph{Maurer-Cartan} if it satisfies the equation 
$$ 0 = m \bullet m$$ 
\begin{prop}
Let $V$ be a $k$-module, an element $m \in \CC(V)^2$ is Maurer-Cartan if and only if $m:V \otimes V \longrightarrow V$ is an associative multiplication on $V$.
\end{prop}
In the next corollary, let $k$ be a field of characteristic $0$. 
\begin{cor}
Let $A= (V,m)$ be an associative $k$-algebra, the $m$-twisted $\mathsf{dgLie}$-algebra
$$(s\inv \CC(A),d_m,[-,-])$$
where
$$d_m(\te) := [m, \te] = m \bullet \te - (-1)^{|\te|} \te \bullet m$$
governs the deformations of $A$ as an associative algebra, i.e. 
$$m' \in  \CC(A)^2\text{ is Maurer-Cartan if and only if }A_{m'}:=(V,m+m')\text{ is an associative }k\text{-algebra.}$$ 
\end{cor}

\subsection{The Gerstenhaber-Schack box operad}
\label{subparGS}
Let $\uuu$ be a small category.
\begin{mydef}\label{defQuiver}
A \emph{quiver $\A$ on $\uuu$} consists of the following
\begin{itemize}
\item for every object $U \in \uuu$ a quiver $\A(U)$, consisting of a set of objects $\Ob(\A(U))$ and $k$-modules of arrows $\A(U)(A,A')$ for every two objects $A,A'$,
\item for every morphism $u:V \longrightarrow U$ in $\uuu$, a function between the respective sets of objects $u\st:\Ob(\A(U)) \longrightarrow \Ob(\A(V))$. 
\end{itemize}
\end{mydef}

\medskip

\noindent \emph{Notations.}
Let $\si = (U_{0} \overset{u_{1}}{\rightarrow} U_{1} \rightarrow  \ldots   \overset{u_{p}}{\rightarrow}  U_{p} )$ be a $p$-simplex in the category $\uuu$, then we have an induced function $\Ob(\A(U_{p})) \longrightarrow \Ob(\A(U_{0}))$ defined as
$$\si\st := u_1 \st \circ \ldots \circ u_p \st$$
Note that for a $0$-simplex $\si=U$, we set $\si\st$ as the identity function $\Id_{\Ob(\A(U))}$.
Let $\si_{\leq t}:= (u_1,\ldots,u_t)$ be the first $t$ arrows of $\si$, and $\si_{>t} := (u_{p-t+1},\ldots,u_p)$ the last $t$ arrows of $\si$.

Let $\Delta_{f}^{inj} \subset \Delta$ be the subcategory of the simplex category containing the strictly increasing order morphisms preserving the end points. It acts on the nerve $N(\uuu)$, that is, for $\si$ a $p$-simplex in $\uuu$ and $f:[r] \hookrightarrow [p]$ a morphism in $\Delta^{inj}_f$, its image $f(\si)$ is a $r$-simplex in $\uuu$.
\medskip 
 
\begin{mydef}
For $\A$ a quiver over $\uuu$, and $p,q,r \geq 0$, we define a $k$-module $\CC_\square^{p,q,r}(\A)$. \newline For $p\geq r$, define
$$\CC_\square^{p,q,r}(\A) := \prod_{\substack{ \si:V \rightarrow U \; p\text{-simplex} \\ f\in \Delta^{inj}_f(r,p)}}\prod_{A_0,\ldots,A_q \in \A(U)} \Hom( \bigotimes_{i=1}^{q} \A(U)(A_{i-1},A_{i}), \A(V)(\si\st A_0, f(\si) \st A_q) ) .$$
For $p<r$, define $\CC_\square^{p,q,r}(\A) = 0$.
\end{mydef}

An element $\te \in \CC_\square^{p,q,r}(\A)$ can naturally be drawn as a box
$$\scalebox{0.8}{$\input{ExtendedGSElements.tikz}$}$$
Hence, $\CC_\square^{p,q,r}(\A)$ will have a box operad structure, which we now make explicit.

As an auxiliary intermediate step, we define a horizontal composition of elements in $\CC_\square^{\bu,\bu,\bu}(\A)$ by tensoring over $\A$. Indeed, consider $\te \in \CC_\square^{p,q,r}(\A)$ and $\te' \in \CC_\square^{r,q',t}(\A)$, then we define a set of elements $\te \otimes_{\A} \te'$. Let $\si \in N_p(\uuu)$, $f\in \Delta_f^{inj}(t,p)$ and $A=(A_0,\ldots,A_{q+q'})$, then we define
$$(\te \otimes_\A \te')^{\si,f}_A :=  \sum_{\substack{  f_1 \in \Delta_f^{inj}(r,p) \\  f_2 \in \Delta_f^{inj}(t,r) \\ f= f_1\circ f_2}} \te^{\si,f_1}_{(A_0,\ldots,A_q)} \otimes \te^{'f_1(\si),f_2}_{(A_{q+1},\ldots,A_{q+q'})}$$
pictured as
$$ \scalebox{0.8}{$\input{horizontalCompositionGS.tikz}$} $$ 
Note that this set of elements is not a part of $\CC_\square^{\bu,\bu,\bu}(\A)$. 
\begin{prop}
$\CC_\square^{\bu,\bu,\bu}(\A)$ comes equipped with a box operad structure:
\begin{itemize}
\item the composition map
$$\circ: \mathrm{\CC}_\square^{p,n,r}(\A) \otimes \CC_\square^{p_0,q_1,p_1}(\A) \otimes \ldots \otimes \CC_\square^{p_{n-1},q_n,p_n}(\A) \longrightarrow \CC_\square^{p+p_0,\sum_{i=1}^n q_i, r+p_n}(\A)$$
 composes morphisms of quivers after tensoring over $\A$, i.e. \newline
for respective elements $\te,\te_1,\ldots,\te_n \in \CC_\square^{\bu,\bu,\bu}(\A)$, a simplex $\si \in N_{p+p_0}(\A)$ and a function $f\in \Delta_f^{inj}(r+p_n,p+p_0)$ such that $f(r) = p$, we obtain
$$(\te \circ (\te_1,\ldots,\te_n))^{\si,f} := \te^{ \si_{<p}, f_{\leq r}} \circ ( \te_1 \otimes_{\A} \ldots \otimes_{\A} \te_n)^{\si_{>p},f_>r}$$
pictured as
$$\scalebox{0.75}{$\input{boxoperadGS.tikz}$}$$
where $f_{\leq r} $ and $f_{>r}$ are the restrictions of $f$ to respectively $[r]\sub [r+p_0]$ and $[r+p_0]\setminus [r]$. Otherwise, the composition is zero.
\item the identity map $\Id \in \CC_\square^{0,1,0}(\A)$ is the unit.
\end{itemize} 
We call $\CC_\square^{\bu,\bu,\bu}(\A)$ the \emph{Gerstenhaber-Schack box operad}.
\end{prop}
\begin{proof}
The associativity of the composition is readily verified by the drawings.
\end{proof}

\begin{cor}
For $\A$ a quiver on $\uuu$, $\Tot(s_{\thin}\inv \CC_\square^{\bu,\bu,\bu}(\A))$ is an $\Linf$-admissible algebra.
\end{cor}

%\subsection{Lax prestacks and the Gerstenhaber-Schack object}

%A \emph{lax prestack} is a (non-unital) lax functor taking values in $k$-linear categories. 
%A prestack is a lax functor which is moreover a pseudofunctor.
 %Note that we will in general work without units and thus omit writing `non-unital'.
%
% 
%\begin{mydef}
%For $\A$ a $k$-quiver on $\uuu$ and $p,q\geq 0$, define the $k$-module
%$$\CC_{GS}^{p,q}(\A) := \prod_{\substack{ \si:V \rightarrow U \;\\ p\text{-simplex}}}\prod_{A_0,\ldots,A_q \in \A(U)} \Hom( \bigotimes_{i=1}^{q} \A(U)(A_{i-1},A_{i}), \A(V)(\si\st A_0, |si|\st A_q) )$$
%where $|-|: N(\uuu) \longrightarrow \uuu$ is the evaluation map. We define its\emph{ Gerstenhaber-Schack object} as the totalisation
%$$\CGS^{n}(\A) = \bigoplus_{p+q=n}\CC^{p,q}(\A)$$
%\end{mydef}

\subsubsection{The Gerstenhaber-Schack object}

The components $\CC_{GS}^{\bullet,\bullet}(\A)$ of the Gerstenhaber-Schack object $\CGS(\A)$ \cite{DVL} correspond to specific components of the GS box operad. Indeed, for $q\geq 0$ we have
$$\CC^{0,q}_{GS}(\A) = \CC_{\square}^{0,q,0}(\A)$$
and for $p\geq 1$ we have
$$\CC^{p,q}_{GS}(\A) = \CC_{\square}^{p,q,1}(\A).$$ 
Explicitly, we obtain
$$\CC_{GS}^{p,q}(\A) = \prod_{\substack{ \si:V \rightarrow U \;\\ p\text{-simplex}}}\prod_{A_0,\ldots,A_q \in \A(U)} \Hom( \bigotimes_{i=1}^{q} \A(U)(A_{i-1},A_{i}), \A(V)(\si\hs A_0, |\si| \hs A_q) ) $$
where $|-|: N(\uuu) \longrightarrow \uuu$ composes the simplex in the category $\uuu$. The GS object is then the totalised graded $k$-module
$$\CGS^n(\A) = \bigoplus_{n=p+q}\CC^{p,q}(\A)$$

As a result, an element of $\CGS^{p,q}(\A)$ corresponds to a box of a specific arity $(p,q,\min \{1,p\})$. We call box of this arity a \emph{GS box}.
\begin{prop}\label{propGSLinf}
$s\inv \CGS(\A)$ is an $\Linf$-admissible subalgebra of $\Tot(s_{\td}\inv \CC_\square^{\bu,\bu,\bu}(\A))$.
\end{prop}
\begin{proof}
It suffices to show that the $\Linf$-operations restrict well: given a thin-quadratic stacking $S$ taking values in GS boxes, we show that the resulting box is a GS box.

For type $\rom{1}$ and $\rom{2}$, this is clear from the pictures \ref{type1} and \ref{type2}. Let $S$ be of type $\rom{3}$ whose resulting box has size $(p,q,r)$ (see \ref{type3}). It suffices to show that $r=1$. As $S$ takes value in GS boxes, its non-thin boxes have exactly one single right output. As their corresponding subgraph of the horizontal composite graph $\Hor_S$ is connected, they form a tree. Hence, $r$ corresponds to the output of this tree and thus $r=1$. 
% Indeed, if $S$ is of type $\rom{1}$ or $\rom{2}$, this is clear. If $S$ is of type $\rom{3}$, we have that the non-thin boxes horizontally form a tree $T_S \subseteq \Hh_S$: indeed, the non-thin rectangle form a connected component and they all have a single horizontal output. Hence, the horizontal output $r$ is determined by the right-most rectangle in $S$ (not counting the thin bottom rectangle), that is, the rectangle corresponding to the root of $T_S$. As this is a GS rectangle by assumption, we conclude that $r=1$.
\end{proof}

Let us make these operations as concrete as possible.

\subsubsection{The operation $P_2$}

% In general, given $\te_i \in \CC^{p_i,q_i}(\A)$ for $i=1,\ldots,n$, we have 
%\begin{equation} \label{eqPnsigns}
%P_n(s\inv \te_1,\ldots,s\inv \te_n) = \sum_{S \in \Thin^2(n)} (-1)^{ S^{\td } + \frac{n(n-1)}{2} +1 +n(p+q-1) + \sum_{i=1}^n (p_i+q_i-1)\sum_{j>i}(p_j+q_j-1) }s\inv S(\te_1,\ldots,\te_n)
%\end{equation}
%Note that here we interpret $\te$ also as part of $\CGS(\A)$ without writing $\te^{\td}$ as before.
Employing \eqref{signsStacking} and \eqref{signsQuadraticThin}, we can make some of the signs explicit (using \ref{subapptdsign}). Given, $\te_1 \in \CC^{p_1,q_1}(\A)$ and $\te_2 \in \CC^{p_2,q_2}(\A)$, we have a closed formula for $P_2$:
\begin{itemize}
\item if $p_1=p_2=0$, then
$$P_2(s\inv \te_1,s \inv \te_2) = \sum_{i=1}^{q_1} (-1)^{(i-1)(q_2-1)} s\inv \; \scalebox{0.8}{$\input{typeIGS.tikz}$}$$
 corresponding to the Gerstenhaber brace on the nose;
\item If $p_1>0$ and $p_2=0$, then
$$P_2(s\inv \te_1, s\inv \te_2) = \sum_{i=1}^{q_1} (-1)^{i(q_2-1)} s\inv \; \scalebox{0.8}{$\input{typeII.tikz}$}\; ;$$
\item If $(p_1,q_1) =(0,1)$ and $p_2>0$, then
$$P_2(s\inv \te_1, s\inv \te_2) =  s\inv \; \scalebox{0.8}{$\input{typeIIIp2.tikz}$}\; ;$$
\item Otherwise, $P_2(s\inv \te_1,s\inv \te_2)= 0$.
\end{itemize}
\subsubsection{The operations $P_n$}
For $n>2$ there is no closed formula for $P_n$ as the stackings involved differ greatly. However, we do have the following general formula: given $\te_1 \in \CC^{p_1,q_1}(\A), \ldots, \te_n \in \CC^{p_n,q_n}(\A)$, we have
\begin{equation} \label{eqPnsigns}
P_n(s\inv \te_1,\ldots,s\inv \te_n) = \sum_{S \in \Thin^2_{\rom{3}}(n)} (-1)^{ S^{\td } + \frac{n(n-1)}{2} +1 +n(p+q-1) + \sum_{i=1}^n (p_i+q_i-1)\sum_{j>i}(p_j+q_j-1) }s\inv \; \scalebox{0.8}{$\input{typeIIIpn.tikz}$}
\end{equation}

\subsection{Deforming lax prestacks}\label{subparlax}

\subsubsection{Lax prestacks}

\begin{mydef}\label{deflaxprestack}
A \emph{lax prestack} $\A= (\A,m,f,c)$ over $\uuu$ consists of the following data:
\begin{itemize}
\item for every object $U \in \uuu$, a $k$-linear category $(\A(U), m^U)$ where $m^U$ is the composition of morphisms in $\A(U)$.
\item for every morphism $u:V \longrightarrow U$ in $\uuu$, a $k$-linear functor $f^u = u \st : \A(U) \longrightarrow \A(V)$; these functors are called the \emph{restrictions}.
\item for every couple of morphisms $v: W \longrightarrow V, u: V \longrightarrow U$ in $\uuu$, a $k$-linear natural transformation
$$c^{u,v}: v\st u \st \longrightarrow (uv)\st;$$
these transformations are called the \emph{twists}. Moreover, the twists have to satisfy the following coherence condition for every triple $w:T \longrightarrow W$, $v :W \longrightarrow V$ and $u:V \longrightarrow U$:
$$c^{u,vw}(c^{v,w} \circ 	u \st) = c^{uv,w}(w\st \circ c^{u,v}).$$
\end{itemize}
%A prestack is a lax prestack where the natural transformations $c^{u,v}$ are isomorphisms.
\end{mydef}

A lax prestack $\A$ over $\uuu$ has an underlying quiver $\A$ on $\uuu$ and the associated triple $(m,f,c)$ corresponds to a degree $2$ element of its GS object $\CGS(\A)$.

\subsubsection{Maurer-Cartan elements}
We describe the Maurer-Cartan elements $\al$ of $s\inv \CGS(\A)$. Due to the size constriction of 
$$\al = (m,f,c) \in s\inv\CGS(\A)_1 = \CC^{0,2}(\A) \oplus \CC^{1,1}(\A) \oplus \CC^{2,0}(\A)$$
the Maurer-Cartan equation reduces significantly.

\begin{lemma}\label{MCequationshort}
For $\al \in s\inv\CGS(\A)_1$, we have that $P_n(\al,\ldots,\al) = 0$ for $n\geq 5$.
\end{lemma}
\begin{proof}
Consider a thin-quadratic stacking $S$ with $n\geq 5$ boxes, then it is of type $\rom{3}$. If the right-most non-thin rectangle has size $(1,1,1)$, $S$ consists of at most $3$ boxes. Thus, $S$ is of the form
$$\input{lemmareduce_1.tikz}$$
where the right-most non-thin rectangle has size $(2,0,1)$. In order to fill this up with non-thin boxes of size $(1,1,1)$ or $(2,0,1)$, we only have the following possibilities
$$\input{lemmareduce_2.tikz}, \quad \input{lemmareduce_3.tikz} \quad \text{ and } \quad \input{lemmareduce_4.tikz}$$
As none of these consists of more than $4$ boxes, we have that $P_n = 0$ for $n \geq 5$.
\end{proof}

\begin{theorem}\label{MCprestack}
Let $\A$ be a $k$-quiver, then $\al= (m,f,c) \in \MC\left(s\inv\CGS(\A)\right)$ if and only if $(\A,m,f,c)$ is a lax prestack. 
\end{theorem}
\begin{proof}
We first look at the equation
$$\MC(\al) = -P_2(\al,\al) + P_3(\al,\al,\al) + P_4(\al,\al,\al,\al).$$
By projecting to each summand of
$$\CGS^3(\A) = \CC^{0,3}(\A) \oplus \CC^{1,2}(\A) \oplus \CC^{2,1}(\A) \oplus \CC^{3,0}(\A)$$
we, equivalently, obtain four equations. We show that these four equations correspond respectively to the associativity of $m$, to the restriction maps $f$ preserving the compositions $m$, to the naturality of the transformation $c$ and to the cocycle condition for $c$. It then follows directly that $\MC(\al) = 0$ if and only if $(m,f,c)$ is a lax prestack structure on $\A$.

First, as $\al$ has degree $1$, we note that the signs in \eqref{eqPnsigns} reduce to 
$$P_n(\al,\ldots,\al) = \sum_{S \in \Thin^2(n)} (-1)^{ S^{\td} + 1 } S(\al,\ldots,\al)$$
%Moreover, note that $(-1)^{ S^{\td}} = (-1)^{ S^{\dd}+2} = (-1)^{ S^{\dd}}$ for stackings of type $\rom{1}$ and $\rom{2}$, and $(-1)^{ S^{\td}} = (-1)^{ S^{\dd} + n+  \frac{n(n-1)}{2} }= (-1)^{ S^{\dd} + \frac{n(n+1)}{2}}$ for stackings of type $\rom{3}$.
We relegate the computations of $(-1)^{S^{\td}}$ for the corresponding stackings to \S \ref{signsMC}.

Projecting to $\CC^{p,q}(\A)$, we obtain for $(p,q)=(0,3)$
\begin{align*}
\MC(\al)_{[0,3]} = -P_2(\al,\al)_{[0,3]} = - \quad \input{MC03_1.tikz} \quad + \quad \input{MC03_2.tikz}
\end{align*}
Thus, for $U \in \uuu$ and $(A_0,A_1,A_2,A_3)$ in $\A(U)$ we have 
\begin{equation*}
m^U_{A_0,A_2,A_3}  \circ (m^U_{A_0,A_1,A_2} \otimes \A(U)) = m^U_{A_0,A_1,A_3}\circ ( \A(U) \otimes m^U_{A_1,A_2,A_3}), \label{assoc_data}
\end{equation*} 
Hence, $(\A(U),m^U)$ defines a category. 

%with signs computed as
%$$(-1)^{1+ 1+ C^{0,0,0;0,0}_{2,1} + \eta } = (-1)^{1+1+1+0} = (-1)^{3} \quad \text{ and } \quad (-1)^{ 1+1+ C^{0,0,0;0,0}_{1,2} + \eta} = (-1)^{1+1+0+0} = (-1)^{2} $$
For $(p,q) = (1,2)$, we compute
\begin{align*}
\MC(\al)_{[1,2]} = -P_2^{GS}(\al,\al)_{[1,2]} +P_3^{GS}(\al,\al,\al)_{[1,2]} =  \quad \input{MC12_2.tikz} \quad - 	 \quad \input{MC12_1.tikz}
\end{align*}
Thus, for $V \overset{u}{\rightarrow} U$ in $\uuu$ and $A_0,A_1,A_2$ in $\A(U)$ we have
\begin{equation*}
f^u_{u\st A_0,u\st A_2} \circ m^U_{A_0,A_1,A_2} = m^V_{u\st A_0,u\st A_1,u\st A_2} \circ ( f^u_{A_0,A_1} \otimes f^u_{A_1,A_2}), \label{functorial_data}
\end{equation*}
Hence, the restriction map $f^u:\A(U) \longrightarrow \A(V)$ defines a functor.

%with signs computed as
%$$ (-1)^{ 1+1+ C^{0,0;1,1}_{2} } = (-1)^{1+1+0} =(-1)^2 \quad \text{ and } \quad (-1)^{1+ C^{1,1,1;0,0}_{1,1} + \frac{3(3+1)}{2}} = (-1)^{1+0+6} = (-1)^7 $$
For $(p,q) = (2,1)$, we compute
\begin{align*}
\MC(\al)_{[2,1]} = P_3^{GS}(\al,\al,\al)_{[2,1]} + P_4^{GS}(\al,\al,\al,\al)_{[2,1]} =  \quad \input{MC21_2.tikz} \quad - \quad \input{MC21_1.tikz}
\end{align*}
Thus, for $W \overset{v}{\rightarrow} V \overset{u}{\rightarrow} U$ in $\uuu$ and $A_0,A_1$ in $\A(U)$ we have
\begin{equation*}
m^W_{v\st u \st A_0, v\st u\st A_1, (uv)\st A_1} \circ ( (f^v_{u\st A_0, u \st A_1} \circ f^u_{A_0,A_1}) \otimes c^{u,v}_{A_1} ) = m^W_{v\st u \st A_0, (uv)\st A_0, (uv)\st A_1} \circ ( c^{u,v}_{A_0} \otimes f^{uv}_{A_0,A_1}) \label{naturality_data}
\end{equation*}
Hence, the twist $c^{u,v}$ defines a natural transformation $f^v f^u \longrightarrow f^{uv}$. 
%with signs computed as
%$$ (-1)^{ 1+ C^{2,1,1;0,0}_{0,1} + \frac{3(3+1)}{2}} = (-1)^{ 1-1 + 6} = (-1)^{6} \quad \text{ and } \quad (-1)^{ 1+ C^{2,2,1;0,0}_{1,0} + C^{1,1;1,1}_{1} + \frac{4(4+1)}{2} } = (-1)^{ 1+0+2+10} =(-1)^{13}$$
%\begin{center}
%\begin{tabular}{c| c | c}
%\multicolumn{1}{c}{} & \multicolumn{1}{c}{$\tikzfig{MC21_1_sign}$} & \multicolumn{1}{c}{}  \\ \hline
%$X^v$ & $0_3 0_2 1_2 1_1 \leadsto 1_1 0_2 1_2 0_3 $ & $5$ \\
%$X^h$ & $B^{0,0}_{2,1} + D_{2,1,1}$ & $0$ \\
%$\chi$ & $301~010 \leadsto 30~01~10$ & $1$\\
%$\textcolor{red}{\frac{n(n-1)}{2}}$ & & $\textcolor{red}{3}$
%\end{tabular} $\quad$ \begin{tabular}{c| c | c}
%\multicolumn{1}{c}{} & \multicolumn{1}{c}{$\tikzfig{MC21_2_sign}$} & \multicolumn{1}{c}{}  \\ \hline
%$X^v$ & $0_4 0_3 0_2 1_1 1_4 \leadsto 1_1 0_2 0_3 0_4 1_4$ & $6$ \\
%$X^h$ & $B^{2,1}_{0,0} + D_{2,2,1} + B_{1,1}^{1,1}$ & $2$ \\
%$\chi$ & $3110~0001 \leadsto 30~10~10~01$& $0$\\
%$\textcolor{red}{\frac{n(n-1)}{2}}$ & & $\textcolor{red}{6}$
%\end{tabular} 
%\end{center}

For $(p,q) = (3,0)$, we compute
\begin{align*}
\MC(\al)_{[3,0]} = P_3^{GS}(\al,\al,\al)_{[3,0]} + P_4^{GS}(\al,\al,\al,\al)_{[3,0]} = - \quad \input{MC30_1.tikz} \quad + \quad \input{MC30_2.tikz}
\end{align*}
Thus, for $T \overset{w}{\rightarrow} W \overset{v}{\rightarrow} V \overset{u}{\rightarrow} U$ in $\uuu$ and $A_0$ in $\A(U)$ we have
\begin{multline*}
 m^T_{w\st v\st u\st A_0, (vw)\st u\st A_0, (uvw)\st A_0} \circ ( c^{v,w}_{u\st A_0} \otimes c^{u,wv}_{A_0} ) =\\ m^T_{w\st v\st u\st A_0, w \st (uv)\st A_0, (wvu)\st A_0} \circ ( (f^w_{v\st u \st A_0, (uv)\st A_0} \circ c^{u,v}_{A_0}) \otimes c^{vu,w}_{A_0} ) \label{cocycle_data} 
\end{multline*}
Hence, the twists $c$ satisfy the cocycle condition.
%with signs computed as
%$$ (-1)^{ 1+ C_{0,0}^{3,2,1;0,0} + \eta_h + \frac{3(3+1)}{2}} = (-1)^{1-2+0+6} =(-1)^{5} \quad \text{ and } \quad (-1)^{ 1+C_{0,0}^{3,2,1;0,0} + C_{0}^{2,1;1,1}+ \frac{4(4+1)}{2} } = (-1)^{1-2+3+10} =(-1)^{12}$$
%\begin{center}
%\begin{tabular}{c| c | c}
%\multicolumn{1}{c}{} & \multicolumn{1}{c}{$\tikzfig{MC30_1_sign}$} & \multicolumn{1}{c}{}  \\ \hline
%$X^v$ & $0_3 0_2 1_2 1_1 1_3 \leadsto 1_1 0_2 1_2 0_3 1_3 $ & $5$ \\
%$X^h$ & $B^{0,0}_{3,1} + D_{3,2,1} + B_{2,1}^{1,1}$ & $3$ \\
%$\chi$ & $300~011 \leadsto 30~01~01$ & $0$\\
%$\textcolor{red}{\frac{n(n-1)}{2}}$ & & $\textcolor{red}{3}$
%\end{tabular} $\quad$ \begin{tabular}{c| c | c}
%\multicolumn{1}{c}{} & \multicolumn{1}{c}{$\tikzfig{MC30_2_sign}$} & \multicolumn{1}{c}{}  \\ \hline
%$X^v$ & $0_4 0_3 0_2 1_3 1_1 1_4 \leadsto 1_1 0_2 0_3 1_3 0_4 1_4$ & $8$ \\
%$X^h$ & $B^{3,1}_{0,0} + D_{3,2,1} + B_{1,1}^{2,1}$ & $4$ \\
%$\chi$ & $3100~0011 \leadsto 30~10~01~01$& $0$\\
%$\textcolor{red}{\frac{n(n-1)}{2}}$ & & $\textcolor{red}{6}$
%\end{tabular} 
%\end{center}
Thus, we deduce that $\MC(\al) = 0$ if and only if 
$(\A,m,f,c)$ is a lax prestack over $\uuu$.
\end{proof}
%If we demand moreover that the twists $c^{u,v}$ are isomorphisms, we obtain that $(\A,m,f,c)$ is a prestack, that is, a pseudofunctor $\uuu^{op} \longrightarrow \Cat(k)$.

\subsubsection{The Deformation complex}

Let $k$ be a field of characteristic $0$. The theory of Maurer-Cartan elements allows to twist the original $\Linf$-algebra by a Maurer-Cartan element $\al$ to obtain an $\Linf$-structure complex governing the deformations of $\al$ as a Maurer-Cartan element \cite[Proposition 17]{markl2010} \cite[Theorem 2.6.1]{merkulov2000}.

We apply this to lax prestacks.
\begin{theorem}
Let $(\A,m,f,c)$ be a lax prestack and write $\al=(m,f,c)\in s\inv \CGS(\A)_1$, then  the operations $L_n^\al$, defined on $\te_1,\ldots,\te_n \in s\inv \CGS(\A)$ as
%$$L_n^\al (\te_1,\ldots,\te_n) := \sum_{r \geq 0} \frac{(-1)^{rn+\frac{r(r+1)}{2}}}{r!} L_{n+r}(\underbrace{\al,\ldots,\al}_{r\text{-times}},\te_1,\ldots,\te_n) $$
$$L_n^\al(\te_1,\ldots,\te_n) := \sum_{r \geq 0}\frac{(-1)^{rn + \frac{r(r+1)}{2}}}{n!} L_{n+r}(\underbrace{\al,\ldots,\al}_{r\text{-times}},\te_1,\ldots,\te_n)$$
%(\underbrace{\al,\ldots,\al}_{r\text{-times}},\te_1,\ldots,\te_n)$$
define a $\Linf$-algebra on the graded $k$-module $s\inv \CGS(\A)$.

In particular, for $\al'=(m',f',c') \in s\inv \CGS(\A)_1$ we have that $\al'$ is a $\MC$-element of $(s\inv \CGS(\A),L^\al_1,L_2^\al,\ldots)$ if and only if $(\A,m+m',f+f',c+c')$ is a lax prestack.
\end{theorem}
\begin{proof}
It suffices to note that $\al'$ is a $\MC$-element if and only if $\al+\al'$ is a MC-element of the original $\Linf$-algebra. Hence, the result follows from Theorem \ref{MCprestack}.
\end{proof}

\appendix

\section{The (thin) desuspended sign}\label{appsign}

A technical ingredient of the higher Gerstenhaber brace (Definition \ref{Linf_theorem}) is an appropriate sign. In this appendix, we give an explicit definition of $(-1)^{S^{\td}}$ for $S \in \boxop$ and complete the proof of Theorems \ref{Linf_theorem} and \ref{MCprestack}.

The sign $(-1)^{S^{\td}}$ is defined in two steps: first, we define the \emph{desuspended sign} $(-1)^{S^{\dd}}$ using generators and relations, and then add to it.

\subsection{The desuspended sign}\label{subappdsign}

Let $\dd: \N^3 \longrightarrow \Z$ be the function $\dd(p,q,r) = p-r + q-1$. 

\begin{constr}\label{signStacking}
For a stacking $S \in \boxop$, we define the \emph{desuspended sign} $(-1)^{S^{\dd}}$ inductively as
\begin{enumerate}
\item for generators $C_{\nth{q}}^{p_0,\ldots,p_n;p,r}$ and $\eta$, we define
$$(-1)^{ C_{\nth{q}}^{p_0,\ldots,p_n;p,r} } = (-1)^{ (p_0+p_n)r+ \sum_{i=1}^n (q_i-1)( n-i)+ \sum_{i=1}^n (q_i-1)(p-r + p_0 - p_{i-1} )} \;, \quad (-1)^{\eta} = 1$$
\item for $S,S' \in \boxop$ composable, we set
$$(-1)^{ (S \circ_i S')^{\dd}} = (-1)^{S^{\dd} + S^{'\dd}}$$
\item for $S\in \boxop$ and $\si \in \Ss_n$, we set
$$(-1)^{ (S^\si)^{\dd}} = (-1)^{S^{\dd} + \si(\dd(p,q,r)) }$$
where $\si(\dd(p,q,r))$ is the Koszul sign associated to permuting elements of degree $\dd(p_i,q_i,r_i)$.
\end{enumerate}
\end{constr}
Observe that the sign of $C_{\nth{q}}^{p_0,\ldots,p_n;p,r}$ is the product of three signs
\begin{align*}
-1)^{ \V_{C_{\nth{q}}^{p_0,\ldots,p_n;p,r}}} = (-1)^{C_{\nth{q};q}} &:= (-1)^{\sum_{i=1}^n (q_i-1)(n-i)},\\
(-1)^{ \Hh_{C_{\nth{q}}^{p_0,\ldots,p_n;p,r}} } 
%= (-1)^{M^{p_0,p_n;p,r}} 
&:= (-1)^{(p_0+p_n)r},\\
(-1)^{ \Sh(C_{\nth{q}}^{p_0,\ldots,p_n;p,r}) } &:= (-1)^{ \sum_{i=1}^n (q_i-1)(p-r + p_0 - p_{i-1} )}
\end{align*}
where the last one is the sign of the following shuffle
$$n,q_1-1,\ldots,q_n-1,p-r,p_0-p_1,\ldots,p_{n-1}-p_n \leadsto n,p-r,q_1-1,p_0-p_1,\ldots,q_n-1,p_{n-1}-p_n$$
We call them respectively the \emph{vertical, horizontal} and \emph{shuffle sign}.
\begin{lemma} 
Construction \ref{signStacking} is well-defined.
\end{lemma}
\begin{proof}
It suffices to verify the relations from $\boxop$. The only relation that requires verifying is the associativity, i.e. we verify that
$$(-1)^{ C_{q^1_1,\ldots,q^n_{m_n}}^{p^1_0,\ldots,p^n_{m_n};p+p_0,r+p_n} + C_{m_1,\ldots,m_n}^{p_0,\ldots,p_n;p,r} } = (-1)^{ C_{\sum_{j=1}^{m_1}q^1_j,\ldots,\sum_{j=1}^{m_n} q^n_j}^{p_0+p^1_0,\ldots,p_n+p^n_{m_n};p,r} + \sum_{i=1}^n C_{q^i_1,\ldots,q^i_{m_i}}^{p^i_0,\ldots,p^i_{m_i};p_{i-1},p_i} + \si(\dd(p,q,r))} $$
where $\si$ is the permutation described in \ref{defgenFc}.
%
%Working modulo $2$, we first compute that
%\begin{align*}
%&C_{q^1_1,\ldots,q^n_{m_n}} + C_{m_1,\ldots,m_n} +  C_{\sum_{j=1}^{m_1}q^1_j,\ldots,\sum_{j=1}^{m_n} q^n_j} + \sum_{i=1}^n C_{q^i_1,\ldots,q^i_{m_i}} \\
%&= \sum_{i=1}^n \sum_{j=1}^{m_i} (q_j^i -1)(\sum_{s=1}^n m_s - \sum_{s<i}m_s -j) + \sum_{i=1}^n (m_i-1)(n-i) + \sum_{i=1}^n (\sum_{j=1}^{m_i}q^i_j-1)(n-i)+ \sum_{i=1}^n \sum_{j=1}^{m_i} (q^i_j-1)(m_i-j)  \\
%&= \sum_{i=1}^n \sum_{j=1}^{m_i} (q_j^i -1)( n-i + \sum_{s>i}m_s)) = \sum_{i=1}^n \sum_{j=1}^{m_i} (q_j^i -1) \sum_{s>i}(m_s-1) = \si(q-1) 
%\end{align*}
We first compute the vertical sign
\begin{align*}
&(-1)^{C_{q^1_1,\ldots,q^n_{m_n}} + C_{m_1,\ldots,m_n} +  C_{\sum_{j=1}^{m_1}q^1_j,\ldots,\sum_{j=1}^{m_n} q^n_j} + \sum_{i=1}^n C_{q^i_1,\ldots,q^i_{m_i}}} \\
&= (-1)^{\sum_{i=1}^n \sum_{j=1}^{m_i} (q_j^i -1)(\sum_{s=1}^n m_s - \sum_{s<i}m_s -j) + \sum_{i=1}^n (m_i-1)(n-i) + \sum_{i=1}^n (\sum_{j=1}^{m_i}q^i_j-1)(n-i)+ \sum_{i=1}^n \sum_{j=1}^{m_i} (q^i_j-1)(m_i-j) } \\
&= (-1)^{\sum_{i=1}^n \sum_{j=1}^{m_i} (q_j^i -1)( n-i + \sum_{s>i}m_s))} = (-1)^{\sum_{i=1}^n \sum_{j=1}^{m_i} (q_j^i -1) \sum_{s>i}(m_s-1)}= (-1)^{\si(\dd(0,q,0))} 
\end{align*}

Next, we compute the horizontal sign
%\begin{align*}
%&M^{p_0^1,p^n_{m_n};p+p_0, r+p_n} + M^{p_0,p_n;p,r} + M^{p_0+p_0^1,p_n+p^n_{m_n};p,r} + \sum_{i=1}^n M^{p^i_0,p^i_{m_i};p_{i-1},p_i} \\
%&= (p_0^1 + p_{m_n}^n)(r+p_n) + (p_0 + p_n)r + (p_0+ p_0^1 + p_n + p_{m_n}^n )r + \sum_{i=1}^n (p_0^i + p^i_{m_i})p_i  \\
%&= (p^1_0 + p^n_{m_n})p_n + \sum_{i=1}^n (p_0^i + p^i_{m_i})p_i = \si(p-r)
%\end{align*}
\begin{align*}
&(-1)^{\Hh_{C_{q^1_1,\ldots,q^n_{m_n}}^{p^1_0,\ldots,p^n_{m_n};p+p_0,r+p_n}} + \Hh_{C_{m_1,\ldots,m_n}^{p_0,\ldots,p_n;p,r}} + \Hh_{C_{\sum_{j=1}^{m_1}q^1_j,\ldots,\sum_{j=1}^{m_n} q^n_j}^{p_0+p^1_0,\ldots,p_n+p^n_{m_n};p,r}} + \sum_{i=1}^n \Hh_{C_{q^i_1,\ldots,q^i_{m_i}}^{p^i_0,\ldots,p^i_{m_i};p_{i-1},p_i}} } \\
&= (-1)^{(p_0^1 + p_{m_n}^n)(r+p_n) + (p_0 + p_n)r + (p_0+ p_0^1 + p_n + p_{m_n}^n )r + \sum_{i=1}^n (p_0^i + p^i_{m_i})p_i}  \\
&= (-1)^{(p^1_0 + p^n_{m_n})p_n + \sum_{i=1}^n (p_0^i + p^i_{m_i})p_i } = (-1)^{\si(\dd(p,0,r))}
\end{align*}

We now show that the signs of the shuffles add up to $(-1)^{ \si(\dd(0,q,0))+\si(\dd(p,0,r)) + \si(\dd(p,q,r))}$. This follows directly from the properties of the shuffle sign: let us first define the sign $(-1)^{\Sh(S)}$ for a general stacking $S \in \boxop((a_1,b_1,c_1),\ldots,(a_n,b_n,c_n);(a,b,c))$ as the sign of the shuffle
$$a_1-1,\ldots,a_n-1,b_1-c_1,\ldots,b_n-c_n \leadsto a_1-1,b_1-c_1,\ldots,a_n-1,b_n-c_n$$
It is then easy to see that $(-1)^{\Sh(S \circ_i S')} = (-1)^{\Sh(S)+\Sh(S')}$ for another stacking $S'$: indeed, this follows from the two equalities $a-1= \sum_{i=1}^n a_i-1$ and $b-c =  \sum_{i=1}^n b_i-c_i$. Moreover, we have for every permutation $\tau$ that $(-1)^{\Sh(S^\tau)} = (-1)^{\tau(a-1)+\tau(b-c) + \tau(a-1+b-c)}$. Indeed, we compute $(-1)^{\Sh(S)}$ alternatively as
\begin{align*}
a_1-1,\ldots,a_n-1,b_1-c_1,\ldots,b_n-c_n &\overset{\tau(a-1)+\tau(b-c)}{\leadsto} a_{\tau(1)}-1,\ldots,a_{\tau(n)}-1,b_{\tau(1)}-c_{\tau(1)},\ldots,b_{\tau(n)}-c_{\tau(n)}\\
&\overset{\Sh(S^\tau)}{\leadsto} a_{\tau(1)}-1,b_{\tau(1)}-c_{\tau(1)},\ldots,a_{\tau(n)}-1,b_{\tau(n)}-c_{\tau(n)} \\
&\overset{\tau(a-1+b-c)}{\leadsto}
 a_1-1,b_1-c_1,\ldots,a_n-1,b_n-c_n 
\end{align*}
This finishes the proof.
\end{proof}

\begin{vb}
For $S$ a thin-quadratic stacking of type $\rom{1}$ or $\rom{2}$, we respectively have the signs
$$(-1)^{(k-i)(l-1)} \quad \text{ and } \quad (-1)^{(k-i+p-r)(l-1)}$$
Note that these agree with the signs for the Gerstenhaber brace. Stackings of type $\rom{3}$ are a lot more complex, thus do not have an easy closed formula for its sign.
\end{vb}

\subsection{The thin desuspended sign}\label{subapptdsign}

We alter the desuspended sign in order to accomodate the properties of the signs appearing in the compiled operad $\boxop_{\td}$.

\begin{mydef}
Let $S\in \boxop((p_1,q_1,r_1),\ldots,(p_n,q_n,r_n;(p,q,r))$ be a stacking. Let $i_1 \downarrow \ldots \downarrow i_k$ be its non-thin boxes ordered by occurrence in the stacking, then we define its \emph{thin desuspended sign} as
$$(-1)^{S^{\td}} = (-1)^{S^{\dd} + n+\sum_{t=1}^n \td(p_t,q_t,r_t)|\{j: i_j >t \}|} $$ 
\end{mydef}
\begin{vb}
Let $S$ be a thin-quadratic stacking of type $I$ or $II$, then their thin desuspended and desuspended signs coincide. If $S$ is of type $III$, we obtain $(-1)^{S^{\td}} = (-1)^{S^{\dd} + n+ \sum_{i=1}^n (n-i)\td(p_i,q_i,r_i)}$.
\end{vb}

In order to neatly state the behaviour of this new sign under composition, we first define an auxiliary sign.

\begin{constr}\label{constrTDextrasign}
Let $S,S'\in \boxop$ be composable stackings and assume $S'$ has a non-thin rectangle. Let
\begin{itemize}
\item $s_1 \downarrow \ldots \downarrow s_k$ be the non-thin boxes in $S$,
\item $s'_1\downarrow \ldots \downarrow s'_{k'}$ be the non-thin boxes in $S'$,
\item $s''_1 \downarrow \ldots \downarrow s''_{k''}$ be the non-thin boxes in $S''$,
\item $\al$ be the function $\al(j)= j+i-1$ identifying boxes $s'_j$ with their image in $\al(s'_j)$ in $S \circ_i S'$,
\item $\be$ be the function $\be(j) = j$ if $j<i$, $\be(i)=\al(s'_1)$ and $\be(j)=j+m-1$ otherwise, identifying $s_j$ with their image $\be(s_j)$ in $S\circ_i S'$. Note that we identify the rectangle $i$ with the first first non-thin rectangle $s_1'$ of $S'$.
\end{itemize}
Then we define the sign $\td(S,S',i)$ as the sign of the shuffle
$$\be(s_1)\ldots \be(s_k) \al(s'_2)\ldots \al(s'_{k'})  \leadsto s''_1\ldots s''_{k''}$$
\end{constr}

\begin{prop}\label{propsigns}
 For a stacking $S \in \boxop((p_1,q_1,r_1),\ldots,(p_n,q_n,r_n);(p,q,r))$, we have 
$$(-1)^{(S^\si)^{\td}} = (-1)^{S^{\td} + \si(\td(\underline{p},\underline{q},\underline{r}))}  $$
and for two stackings $S,S'\in \boxop$, we have
$$(-1)^{(S \circ_ i S')^{\td}} = (-1)^{S^{\td}+S^{'\td} + |S^{'\td}|\sum_{j<i}\td(p_j,q_j,r_j)+ \td(S,S',i) +1} $$
\end{prop}
\begin{proof}
We can compute the sign $(-1)^{S^{\td}}$ as the sign $(-1)^{S^{\dd}+n}$ and the sign of the following shuffle
\begin{align*}
0_{i_k}\ldots 0_{i_1} \dd(p_1,q_1,r_1)\ldots \dd(p_n,q_n,r_n) \leadsto \td(p_1,q_1,r_1)\ldots \td(p_n,q_n,r_n)
\end{align*}
where we shuffled the letter $0_{i_j}$ in front of $\dd(p_{i_j},q_{i_j},r_{i_j})$. It is then clear that $(-1)^{(S^{\si})^{\td} } = (-1)^{ S^{\td} + \si(\td(p,q,r))}$ as it simply permutes the numbers $\td(p_i,q_i,r_i)$.

We now show the second equality when composing stackings. There is nothing to show if the resulting rectangle of $S'$ is thin as all boxes in $S'$ are then thin. Assume $S'\in \boxop((p'_1,q'_1,r_1'),\ldots,(p'_m,q'_{m},r'_{m});(p_i,q_i,r_i))$ and let $s_t, s'_t, s''_t$ as in construction \ref{constrTDextrasign},
%\begin{itemize}
%\item $s_1 \downarrow \ldots \downarrow s_k$ be the non-thin boxes in $S$,
%\item $s'_1\downarrow \ldots \downarrow s'_{k'}$ be the non-thin boxes in $S'$,
%\item $s''_1 \downarrow \ldots \downarrow s''_{k''}$ be the non-thin boxes in $S''$.
%\end{itemize}
then we compute the shuffle computing $(-1)^{(S\circ_i S')^{\td}- (S \circ_i S')^{\dd}+n+n'-1}$ as
\begin{align*}
&0_{s_{k''}''}\ldots 0_{s_1''}\dd(p_1,q_1,r_1)\ldots \dd(p'_1,q'_1,r'_1)\ldots\dd(p'_m,q'_m,r'_m) \ldots \dd(p_n,q_n,r_n) \\
&\overset{(1)}{\leadsto} 0_{s'_{k'}}\ldots 0_{s'_2} 0_{s_k}\ldots 0_{s'_1}\ldots 0_{s_1} \dd(p_1,q_1,r_1)\ldots \dd(p'_1,q'_1,r'_1)\ldots\dd(p'_m,q'_m,r'_m) \ldots \dd(p_n,q_n,r_n) \\
&\overset{(2)}{\leadsto} 0_{s'_{k'}}\ldots 0_{s'_2} \td(p_1,q_1,r_1)\ldots 0_{s'_1}\dd(p'_1,q'_1,r'_1)\ldots\dd(p'_m,q'_m,r'_m) \ldots \td(p_n,q_n,r_n) \\
&\overset{(3)}{\leadsto}   \td(p_1,q_1,r_1)\ldots 0_{s'_{k'}}\ldots 0_{s'_2} 0_{s'_1}\dd(p'_1,q'_1,r'_1)\ldots\dd(p'_m,q'_m,r'_m) \ldots \td(p_n,q_n,r_n) \\
&\overset{(4)}{\leadsto} \td(p_1,q_1,r_1)\ldots \td(p'_1,q'_1,r'_1)\ldots\td(p'_m,q'_m,r'_m) \ldots \td(p_n,q_n,r_n)
\end{align*}
where $(2),(3)$ and $(4)$ compute respectively $(-1)^{S^{\td}-S^{\dd}+n}, (-1)^{|S^{'\td}| \sum_{j<i}\td(p_j,q_j,r_j)}$ and $(-1)^{S^{'\td} - S^{'\dd}+n'}$. We can rewrite the shuffle $(1)$ as
\begin{align*}
&0_{s_{k''}''}\ldots 0_{s_1''} \leadsto 0_{s_{1}''}\ldots 0_{s_{k''}''} \leadsto 0_{s_1} \ldots 0_{s'_1} \ldots 0_{s_k} 0_{s'_2} \ldots 0_{s'_{k'}} \leadsto 0_{s'_{k'}}\ldots 0_{s'_2} 0_{s_k}\ldots 0_{s'_1}\ldots 0_{s_1}
\end{align*}
which has sign $(-1)^{ \frac{k''(k''-1)}{2} + \td(S,S',i) + \frac{k(k-1)}{2} + \frac{(k'-1)(k'-2)}{2} + k(k'-1)} = (-1)^{\td(S,S',i)} $ as $k''= k+k'-1$. This completes the proof.
\end{proof}

\subsection{Computing signs}\label{subappcompute}

\subsubsection{Signs for Theorem \ref{Linf_theorem}}

We complete the proof Theorem \ref{Linf_theorem} by computing the appropriate signs.

\begin{lemma}\label{signsLinf}
For every composable pair of thin-quadratic stackings
\begin{itemize}
\item  $S \in \boxop((p_1,q_1,r_1),\ldots,(p_k,q_k,r_k);(p,q,r))$,
\item  $S' \in \boxop((p'_1,q'_1,r_1'),\ldots,(p'_l,q'_{l},r'_{l});(p_i,q_i,r_i))$,
\end{itemize}
and an index $1 \leq i \leq n$, there exists two unique thin-quadratic stackings 
\begin{itemize}
\item  $\tilde{S} \in \boxop((\tilde{p}_1,\tilde{q}_1,\tilde{r}_1),\ldots,(\tilde{p}_{\tilde{k}},\tilde{q}_{\tilde{k}},\tilde{r}_{\tilde{k}});(p,q,r))$,
\item  $\tilde{S}' \in \boxop((\tilde{p}'_1,\tilde{q}'_1,\tilde{r}'_1),\ldots,(\tilde{p}'_{\tilde{l}},\tilde{q}'_{\tilde{l}},\tilde{r}'_{\tilde{l}});(\tilde{p}_j,\tilde{q}_j,\tilde{r}_j))$,
\end{itemize}
an unique index $1\leq j \leq \tilde{k}$ and a unique permutation $\si \in \Ss$ such that 
$$S \circ_i S' = (\tilde{S} \circ_j \tilde{S}')^\si.$$
Moreover, we then have that
$$(-1)^{(k-i)l + i-1 + S^{\td} + S^{'\td} } = (-1)^1 (-1)^{(\tilde{k}-j)\tilde{l} + j-1 + \si + \tilde{S}^{\td} + \tilde{S'}^{\td}}$$
 \end{lemma}
 \begin{proof}
 From the proof of Theorem \ref{Linf_theorem}, we know that we only have to verify the signs. We compute
\begin{align*}
S^{\td} \circ_i S^{'\td} &= (-1)^{|S^{'\td}| \sum_{s<i}\td(p_s,q_s,r_s)} (S\circ_i S')^{\td} \\
&=(-1)^{|S^{'\td}| \sum_{s<i}\td(p_s,q_s,r_s) + \si(\td(p,q,r))} ((\tilde{S} \circ_j \tilde{S'})^{\td})^{\si} \\
&= (-1)^{|S^{'\td}| \sum_{s<i}\td(p_s,q_s,r_s) + \si(\td(p,q,r))+ |\tilde{S}^{'\td}| \sum_{s<j}\td(\tilde{p}_s,\tilde{q}_s,\tilde{r}_s)} (\tilde{S}^{\td} \circ_j \tilde{S'}^{\td})^{\si} 
\end{align*}
On the other hand, Proposition \ref{propsigns} shows that 
$$(-1)^{S^{\td} + S^{'\td} + |S^{'\td}| \sum_{s<i}\td(p_s,q_s,r_s) + \td(S,S',i) } = (-1)^{ \tilde{S}^{\td} + \tilde{S}^{'\td} + |\tilde{S}^{'\td}| \sum_{s<j}\td(\tilde{p}_s,\tilde{q}_s,\tilde{r}_s) + \si(\td(p,q,r)) + \td(\tilde{S},\tilde{S'},j)}$$
We thus simply have to verify the following signs are opposite
$$(-1)^{(k-i)l + i-1 + \td(S,S',i) } \text{ and } (-1)^{(\tilde{k}-j)\tilde{l} + j-1 + \si + \td(\tilde{S},\tilde{S'},j)}$$
Following its definition \ref{constrTDextrasign}, we observe that if $S'$ is of type $\rom{1}$ or $\rom{2}$, then $\td(S,S',i)=0$. Hence, we can already compute
\begin{center}
\begin{tabular}{c|c|c}
&$(k-1)l+(i-1)(l-1) + \td(S,S',i)$ & $(\tilde{k}-j)\tilde{l} + j-1+ \si + \td(\tilde{S},\tilde{S'},j)$ \\
\hline
\eqref{Linfeq1} & $0$ & $(23)$ \\
\eqref{Linfeq2} & $0$ & $1$ \\
\eqref{Linfeq3} & $0$ & $(23)$ \\
\eqref{Linfeq4} & $0$ & $1$ \\
\eqref{Linfeq5} & $i-1$ & $k + (i+1 \ldots k+1) +\td(\tilde{S},\tilde{S'},1)$ \\
\eqref{Linfeq6} & $i-1$ & $(\tilde{k}-j)\tilde{l} + j-1+ \si + \td(\tilde{S},\tilde{S'},j)$
\end{tabular}
\end{center}
For \eqref{Linfeq5}, the sign $(-1)^{\td(\tilde{S},\tilde{S}',1)}$ is computed by the shuffle $2 \ldots k \leadsto 2 \ldots k$. Hence, the signs in \eqref{Linfeq5} are opposite as $(-1)^{(i+1\ldots k+1)} = (-1)^{k-i}$.

For \eqref{Linfeq6}, we first analyse the permutation $\si$: it decomposes as a shuffle $\si_1$ which shuffles the result of the left-hand composition in standard order
$$\scalebox{0.8}{$\input{shuffl_1.tikz}$} \quad \overset{\si_1}{\leadsto} \quad  \scalebox{0.8}{$\input{shuffl_2.tikz}$}$$
and the shuffle $\si_2$ which shuffles the result of the right-hand composition (without applying $\si$) in standard order
$$\scalebox{0.8}{$\input{shuffl_3.tikz}$} \quad \overset{\si_2}{\leadsto} \quad  \scalebox{0.8}{$\input{shuffl_4.tikz}$} $$
Note that it shuffles the numbering of the boxes in $\tilde{S}'$ and the children of rectangle $j$ in stacking $\tilde{S}$.

We easily compute $(-1)^{\si_1} = (-1)^{j-i+1}$.

To finish, we compute $\td(\tilde{S},\tilde{S'},j)$: we can compute it with the following two shuffles
$$2 \ldots (j-1)(j+1)\underbrace{(j+\tilde{l}) \ldots (\tilde{k}+\tilde{l}-1)}_{\tilde{k}-j\text{ elements}}\underbrace{(j+2) \ldots (j+\tilde{l}-1)}_{\tilde{l}-2\text{ elements}} \leadsto 2 \ldots (j-1)(j+1) \ldots (\tilde{k}+\tilde{l}-1)$$
and
$$ 2 \ldots (j-1)(j+1) \ldots (\tilde{k}+\tilde{l}-1) \leadsto 2 \ldots (j-1) \si_2((j+1)\ldots (\tilde{k}+\tilde{l}-1))$$
Hence, $(-1)^{\td(\tilde{S},\tilde{S'},j)} = (-1)^{(\tilde{k}-j)\tilde{l} + \si_2}$, showing that the signs in \eqref{Linfeq6} are opposite.
 \end{proof}
 
\subsubsection{Signs for Theorem \ref{MCprestack}}\label{signsMC}

We complete the proof of Theorem \ref{MCprestack} by explicitly computing the sign $(-1)^{S^{\td}}$ of the corresponding stackings. We have the following table.

\begin{center}

\begin{tabular}{c | c}
stacking & $(-1)^{ S^{\td}}$ \\
\hline
$\scalebox{0.8}{$\input{MC03_1_sign.tikz}$}$ & $(-1)^{1+ 1+ C^{0,0,0;0,0}_{2,1} + \eta } = (-1)^{1+1+1+0} = (-1)^{3}$ \\
$\scalebox{0.8}{$\input{MC03_2_sign.tikz}$}$ & $(-1)^{ 1+1+ C^{0,0,0;0,0}_{1,2} + \eta} = (-1)^{1+1+0+0} = (-1)^{2}$ \\
$\scalebox{0.8}{$\input{MC12_2_sign.tikz}$}$ & $(-1)^{ 1+1+ C^{0,0;1,1}_{2} } = (-1)^{1+1+0} =(-1)^2$ \\
$\scalebox{0.8}{$\input{MC12_1_sign.tikz}$}$ & $(-1)^{1+ C^{1,1,1;0,0}_{1,1} + \frac{3(3+1)}{2}} = (-1)^{1+0+6} = (-1)^7$ \\
$\scalebox{0.8}{$\input{MC21_2_sign.tikz}$}$ & $(-1)^{ 1+ C^{2,1,1;0,0}_{0,1} + \frac{3(3+1)}{2}} = (-1)^{ 1-1 + 6} = (-1)^{6}$ \\
$\scalebox{0.8}{$\input{MC21_1_sign.tikz}$}$ & $(-1)^{ 1+ C^{2,2,1;0,0}_{1,0} + C^{1,1;1,1}_{1} + \frac{4(4+1)}{2} } = (-1)^{ 1+0+2+10} =(-1)^{13}$   \\
$\scalebox{0.8}{$\input{MC30_1_sign.tikz}$}$ & $(-1)^{ 1+ C_{0,0}^{3,2,1;0,0} + \eta_h + \frac{3(3+1)}{2}} = (-1)^{1-2+0+6} =(-1)^{5} $   \\
$\scalebox{0.8}{$\input{MC30_2_sign.tikz}$}$ & $(-1)^{ 1+C_{0,0}^{3,2,1;0,0} + C_{0}^{2,1;1,1}+ \frac{4(4+1)}{2} } = (-1)^{1-2+3+10} =(-1)^{12}$  
\end{tabular}

\end{center}

\section{Compiling a coloured symmetric operad}\label{appcompile}

Given a coloured symmetric operad $\Oo$ enriched over graded $k$-modules, with set of colours $A$. Let $V$ be a $\Oo$-algebra and $\psi: A \longrightarrow \Z$ a function assigning to each colour an integer. 

In this appendix, we totalize the collection $V$ into a single graded $k$-module $\prod_\psi V$ where elements of $V(a)$ are placed in additional degree $\psi(a)$. Next, we compile $\Oo$ into an (uncoloured) symmetric graded linear operad $\prod_\psi \Oo$ (Proposition \ref{operadstructure}) exactly such that the $\Oo$-algebra structure on $V$ lifts to a $\prod_\psi \Oo$-algebra structure on $\prod_\psi V$ (Corollary \ref{algebralift}). Note that the main intricacy lies in a coherent choice of signs. As we work with the product, this is only well defined if some finiteness (or convergence) property is satisfied (Lemma \ref{compilatenwelldefined}). 

\subsubsection{$z$-Suspension of chain complexes}

Let $C$ be a chain complex of $k$-modules and $z\in \Z$ an integer, then the $z$-suspension is defined as 
$$ s^z C := (ks)^{\otimes z} \otimes C$$
where $s$ is a formal element concentrated in degree $1$.

We have the following properties
\begin{itemize}
\item the differential $s^z C$ alters by $(-1)^z$, i.e. $d_{s^z C}= (-1)^z d_C$,
\item we have an isomorphism of $k$-modules $(s^z C)_n \cong C_{n-z}$,
\item an element in $C$ of degree $n$ has degree $n+z$ in $s^z C$,
\item we have two canonical maps $C \longrightarrow s^zC$ and $s^{z} C \longrightarrow C$ of respective degrees $z$ and $-z$.
\end{itemize}
Note that $z=1$ and $z=-1$ are respectively the suspension and desuspension.

We make the following choice of identification
$$ s\inv  s= Id_V = - s s\inv $$
This is a consequence of the Koszul sign involved in
$$\circ ( s \otimes s \inv ) = - \circ( s \inv \otimes s).$$
As a result, we have 
$$ s^{-z} \circ s^{z} = (-1)^{\frac{z(z-1)}{2}} (s\inv s)^z =  (-1)^{\frac{z(z-1)}{2}} Id_V$$

\subsection{Totalization of a collection}\label{subapptotcoll}

\noindent \emph{Notations.}
Given a set $A$, we abbreviate a sequence $\nth{a}$ as $\underline{a}$.
We have an operadic composition of sequences: $\nth{a} \circ_i \fromto{a'}{m} := a_1,\ldots,\fromto{a'}{m},\ldots,a_n$ replacing $a_i$ by the sequence $\fromto{a'}{m}$. We abbreviate by writing $\underline{a} \circ_i \underline{a'}$. 

\medskip 

We construct a graded $k$-module by shifting the collection $k$-modules along $\psi$
%we haven an induced algebra map
%$$\Tot(\phi): \Tot(\Oo) \longrightarrow \Tot(\End(V)): \Tot(\phi)(x_{\nth{a};a}) = (\phi(x_{\nth{a};a}))_{\nth{a};a}$$ 
$$\prod_\psi V := \prod_{a \in A} s^{\psi(a)} V(a)$$
Note in case that $V(a)$ are graded $k$-modules as well, then this is a product of graded $k$-modules, i.e. we only consider the $k$-linear finite sums of sequences of constant degree. For $v\in V(a)$, we denote $v^\psi := s^{\psi(a)}v$ and thus $|v^\psi| = \psi(a) + |v|$.

For an element $t \in \End(V)$ we have an induced map
$$% https://tikzcd.yichuanshen.de/#N4Igdg9gJgpgziAXAbVABwnAlgFyxMJZABgBpiBdUkANwEMAbAVxiRDgD1gAdbtbABR0A+gEYAlAF8AakLHiABLwh4AtvCXcGUFXE0qs6vZx59BIsFNkXxISaXSZc+QilHkqtRizYne-LCErILsHEAxsPAIiMlFPemZWRBBreX01DV5tXXTDDVTLUMdIlyJ3OOoEn2TrW0lPGCgAc3giUAAzACcIVSQyEBwIJAAmSu8k9i4AWn9zeUlco01snD1lDONp2cCbSRBqBjoAIxgGAAUnKNcQTqwmgAscIpAunpHqQaQAZjHEtnaBAAPWwHY6nC4laLJW4PJ72DrdXqIH4DIaIdxeP7JPxmHZSfYgQ4nc6XUrQu6PZ6vJH9T7o37VEC8ADuWFg9zoOGA7UkAhZbJgHK5gMkdQokiAA
\begin{tikzcd}
s^{\psi(a_1)}V(a_1) \otimes \ldots \otimes s^{\psi(a_n)}V(a_n) \arrow[d, "s^{-\psi(a_1)} \otimes \ldots \otimes s^{-\psi(a_n)}"'] \arrow[r, dotted] & s^{\psi(a)}V(a)                \\
V(a_1) \otimes \ldots \otimes V(a_n) \arrow[r, "t"']                                                                                                                 & V(a) \arrow[u, "s^{\psi(a)}"']
\end{tikzcd}$$
We define an element $t^\psi \in \End( \prod_\psi V )$ by adjusting with a sign $(-1)^{t^\psi}$
$$t^\psi := (-1)^{t^\psi} s^{\psi(a)} t \left(s^{-\psi\left(a_1\right)}  \otimes \ldots \otimes s^{-\psi\left(a_n\right)}\right)$$
where 
$$(-1)^{t^\psi} =  (-1)^{\frac{\psi(a)(\psi(a)-1)}{2} + \sum_{i=1}^n \psi(a_i) \sum_{j>i}\psi(a_j)}$$

\begin{lemma}\label{sign_tpsi}
For $t\in \End\left(V\right)\left(\nth{a};a\right)$ and $t'\in \End\left(V\right)\left(\fromto{a'}{n'};a_i\right)$, we have
$$(-1)^{t^\psi + t^{' \psi} + (t\circ_i t')^\psi} = (-1)^{ \psi(\underline{a'};a_i) \sum_{j=1,\neq i}^k \psi(a_j) + \frac{\psi(a_i)(\psi(a_i)-1)}{2}} $$
\end{lemma}
\begin{proof}
We compute 
\begin{align*}
(-1)^{(t\circ_i t')^\psi} &= (-1)^{ \frac{\psi(a)(\psi(a)-1)}{2}+  \sum_{j<i} \psi(a_j) (\sum_{s>j,\neq i} \psi(a_j) + \sum_{s=1}^{n'} \psi(a'_s) ) + \sum_{j=1}^{n'} \psi(a'_j) (\sum_{s>i}\psi(a_s)+ \sum_{s>j}\psi(a'_s))} \\
& (-1)^{  \sum_{j>i} \psi(a_j)\sum_{s>j} \psi(a_s) } \\
&= (-1)^{ t^{\psi} + \sum_{j<i} \psi(a_j) ( \psi(a_i) + \sum_{s=1}^{n'} \psi(a'_s)) + \psi(a_i)\sum_{s>i} \psi(a_s) + \sum_{j=1}^{n'} \psi(a'_j) \sum_{s>i}\psi(a_s) + t^{'\psi} + \frac{\psi(a_i)(\psi(a_i-1)}{2} } \\
&= (-1)^{ t^\psi + t^{'\psi} + \psi(\underline{a'};a_i) \sum_{j=1,\neq i}^k \psi(a_j) + \frac{\psi(a_i)(\psi(a_i)-1)}{2} }
\end{align*}
\end{proof}

 We compute the corresponding signs of the composition.
\begin{lemma}\label{compiledEnd}
For $t\in \End\left(V\right)\left(\nth{a};a\right)$ and $t'\in \End\left(V\right)\left(\fromto{a'}{n'};a_i\right)$, we obtain 
$$t^\psi \circ_i t^{'\psi} = \left(-1\right)^{ |t'| \sum_{j\neq i} \psi\left(a_j\right) + \psi\left(\underline{a'};a_i\right)\sum_{j<i}\psi\left(a_j\right)} (t \circ_i t')^\psi $$
Moreover, for $\si \in \Ss_n$, we have $\left(t^\psi\right)^\si = \left(-1\right)^{\si\left(\psi\left(\underline{a}\right)\right)} (t^\psi)^\si$.
\end{lemma}
\begin{proof}
First, we compute
\begin{align*}
&t^\psi \circ_i t^{'\psi} \\
&= s^{\psi\left(a\right)} \circ t \circ \left( s^{\psi\left(a_1\right)} \otimes \ldots \otimes s^{\psi\left(a_n\right)}\right) \circ \left(Id^{\otimes\left(i-1\right)} \otimes t^{'\psi} \otimes Id^{\otimes\left(n-i\right)}\right) (-1)^{\eps_1} \\
&=  s^{\psi\left(a\right)} \circ t \circ \left( s^{-\psi\left(a_1\right)} \otimes \ldots \otimes s^{-\psi\left(a_i\right)} \circ t^{'\psi} \otimes \ldots \otimes s^{-\psi\left(a_n\right)}\right) (-1)^{\eps_1+ \eps_2}   \\
&= s^{\psi\left(a\right)} \circ t \circ \left( s^{-\psi\left(a_1\right)} \otimes \ldots \otimes t' \circ \left(s^{-\psi\left(a'_1\right)} \otimes \ldots \otimes s^{-\psi\left(a'_{n'}\right)}\right) \otimes \ldots \otimes s^{-\psi\left(a_n\right)}\right)  (-1)^{\eps_1+\eps_2 + \eps_3} \\
&= s^{\psi\left(a\right)} \circ \left(t \circ_i t'\right) \circ \left( s^{-\psi\left(a_1\right)} \otimes \ldots \otimes s^{-\psi\left(a'_1\right)} \otimes \ldots \otimes s^{-\psi\left(a'_{n'}\right)} \otimes \ldots \otimes s^{-\psi\left(a_n\right)}\right) (-1)^{\eps_1+\eps_2+\eps_3+\eps_4}
\end{align*}
where 
\begin{align*}
\eps_1 &= t^\psi, \quad
\eps_2 =  |t^{'\psi}| \sum_{j>i}|\psi\left(a_j\right)|, \quad
\eps_3 =  t^{' \psi} + \frac{\psi(a_i)(\psi(a_i)-1)}{2},  \quad
\eps_4 = |t'| \sum_{j<i} |\psi\left(a_j\right)|
\end{align*}
%\begin{align*}
%\eps_1 &= \frac{\psi(a)(\psi(a)-1)}{2} + \sum_{j=1}^k \psi(a_j) \sum_{s>j} \psi(a_s) , \quad  \eps_2 = |t^{'\psi}| \sum_{j>i}|\psi\left(a_j\right)|, \\
%\eps_3 &=\frac{\psi(a_i)(\psi(a_i)-1)}{2}+\frac{\psi(a_i)(\psi(a_i)-1)}{2} + \sum_{j=1}^l \psi(a_j') \sum_{s>j} \psi(a_s')  , \quad \eps_4 =  |t'| \sum_{j<i} |\psi\left(a_j\right)|
%\end{align*}
Using lemma \ref{sign_tpsi}, we obtain that 
$$\eps_1+\eps_2+\eps_3+\eps_4 = (t\circ_i t')^\psi + |t^{'\psi}| \sum_{j>i}|\psi\left(a_j\right)| + |t'| \sum_{j<i} |\psi\left(a_j\right)| +  \psi(\underline{a'};a_i) \sum_{j=1,\neq i}^k \psi(a_j)$$
The result then follows as $|t^{'\psi}| = |t'|+ \psi(\underline{a'};a_i)$. 

For a permutation $\si \in \Ss_n$, we first note that 
$$(-1)^{ \sum_{i=1}^n \psi(a_{\si(i)}) \sum_{j>i} \psi(a_{\si(j)}) } = (-1)^{ \sum_{i=1}^n \psi(a_{i}) \sum_{j>i} \psi(a_{j}) }$$
and thus $(-1)^{ t^\psi } = (-1)^{ (t^{\si})^\psi }$. Hence, we can compute
\begin{align*}
\left( t^\psi \right)^\si &= s^{\psi\left(a\right)} \circ t \circ \left( s^{-\psi\left(a_1\right)} \otimes \ldots \otimes s^{-\psi\left(a_n\right)}\right) \circ \si^{-1} (-1)^{ t^{\psi}}\\
&= s^{\psi\left(a\right)} \circ t \circ \si^{-1} \circ  \left( s^{-\psi\left(\si\inv\left(a_1\right)\right)} \otimes \ldots \otimes s^{-\psi\left(\si\inv\left(a_n\right)\right)}\right) \left(-1\right)^{\si\left(\psi\left(\underline{a}\right)\right)+t^{\psi} }\\
&= \left(t^{\si}\right)^\psi\left(-1\right)^{\si\left(\psi\left(\underline{a}\right)\right)}
\end{align*}
\end{proof}

We also write the signs appearing in the application of an element $t\in \End(V)$ on concrete elements.
\begin{lemma}\label{signoperationexplicit}
For an operation $t\in \End(V)(\underline{a};a)$ and elements $v_i \in V(a_i)$, we obtain
$$t^\psi(v_1^\psi,\ldots,v_n^\psi) =(-1)^{ \frac{\psi(\underline{a};a)(\psi(\underline{a};a)+1)}{2} +  \psi(a) \psi(\underline{a};a)  + \sum_{i=1}^n|v_i^\psi| \sum_{j>i} \psi(a_j) } t(v_1,\ldots,v_n)^\psi$$
\end{lemma}
\begin{proof}
We first write out the definition 
$$ t^\psi(v_1^\psi,\ldots,v_n^\psi) := (-1)^{ t^\psi + \sum_{i=1}^n \frac{\psi(a_i)(\psi(a_i)-1)}{2} + \sum_{i=1}^n|v_i^\psi| \sum_{j>i} \psi(a_j) } t(v_1,\ldots,v_n)^\psi $$
Using that $(-1)^{\frac{k(k-1)}{2}+ \frac{l(l-1)}{2}} = (-1)^{ \frac{(k+l)(k+l-1)}{2} + kl}$, we obtain
$$(-1)^{  \sum_{i=1}^n \frac{\psi(a_i)(\psi(a_i)-1)}{2} } = (-1)^{ \frac{(\sum_{i=1}^n \psi(a_i))(\sum_{i=1}^n \psi(a_i) -1)}{2} + \sum_{i=1}^n \psi(a_i) \sum_{j>i} \psi(a_i)}$$
Hence, we have
$$(-1)^{ t_\psi + \sum_{i=1}^n \frac{\psi(a_i)(\psi(a_i)-1)}{2}} = (-1)^{ \frac{\psi(a)(\psi(a)-1)}{2} + \frac{(\sum_{i=1}^n \psi(a_i))(\sum_{i=1}^n \psi(a_i)-1)}{2} } = (-1)^{ \frac{\psi(\underline{a};a)(\psi(\underline{a};a)+1)}{2} +  \psi(a) \psi(\underline{a};a)}.$$
\end{proof}

%where we further compute
%$$ (-1)^{ t^\psi + \sum_{i=1}^n \frac{\psi(a_i)(\psi(a_i)-1)}{2} + \sum_{i=1}^n|v_i^\psi| \sum_{j>i} \psi(a_j)}   = (-1)^{ \frac{\psi(a)(\psi(a)-1)}{2} + \sum_{i=1}^n \frac{\psi(a_i)(\psi(a_i)-1)}{2} + \sum_{i=1}^n|v_i| \sum_{j>i} \psi(a_j) } $$
%Moreover, we have that 

\subsection{Totalized operad}\label{subapptotop}
Given a graded $k$-linear coloured operad $\Oo$ with set of colours $A$ and a function $\psi: A \longrightarrow \Z$ associating to every color an integer. For a sequence of colors $(\nth{a};a)$ we set $\psi(\underline{a};a):= \psi(a) - \sum_{i=1}^n \psi(a_i)$. Note that for this definition we have that 
\begin{equation} \label{additive_psi}
\psi(\underline{a};a) + \psi(\underline{a'};a_i) = \psi(\underline{a} \circ_i \underline{a'} ; a)\end{equation}
We compile $\Oo$ into a graded uncoloured operad shifted by $\psi$ by defining
$$\left(\prod_\psi\Oo\right)(n):= \prod_{\underline{a};a} s^{\psi(\underline{a};a)}\Oo(\underline{a};a)$$ 
Note that it is a product of graded $k$-modules and thus we only consider the $k$-linear finite sums of sequences $(x_{\underline{a};a})_{\underline{a};a}$ of constant degree. For $x \in \Oo(\underline{a};a)$, we write $x^\psi:= s^{\psi(\underline{a};a)} x$ for the respective element in $\left(\prod_\psi \Oo\right)(n)$ concentrated in index $(\underline{a};a)$ and degree $\psi(\underline{a};a)+ |x|$.

In order to lift an $\Oo$-algebra structure on a collection $V$ to some $\prod_\psi \Oo$-algebra structure on the graded $k$-module $\prod_\psi V$, we need to define a graded operad structure on $\prod_\psi \Oo$ with suitable signs. Following lemma \ref{compiledEnd}, we obtain the following proposition.
\begin{prop}\label{operadstructure}
$\prod_\psi\Oo$ has the structure of a graded operad defined as 
\begin{itemize}
\item \underline{composition}: for $x= \left(x_{\underline{a};a}^\psi\right)$ and $y = \left(y_{\underline{a};a}^\psi\right)$ its composition $x \circ_i y$ in $\prod_\psi\Oo$ takes for index $(\underline{c};c)$ the value
$$ \sum_{ \underline{a} \circ_i \underline{a'} = \underline{c} } (-1)^{|y| \sum_{j\neq i} \psi(a_j) + \psi(\underline{a'};a_i) \sum_{j<i}\psi(a_j) } (x_{\underline{a};c} \circ_i x_{\underline{a'};a_i})^\psi  $$  
\item \underline{$\Ss$-action}: $\left(x_{\underline{a};a}\right)_{\underline{a};a}^{\sigma}  =  \left((x_{\underline{a};a}^\si)^\psi (-1)^{\si\left(\psi\left(\underline{a}\right)\right)}\right)_{\underline{a};a}$ where the sign is the Koszul sign for permuting elements of degree $\psi(a_i)$. 
\item \underline{unit}: the unit $1$ is zero except for index $a;a$ for which it is $1_a$.
\end{itemize}
\end{prop}
\begin{proof}
First we note that the composition has degree $0$ due to \eqref{additive_psi}. 

Next, we verify the two operad axioms. For $x\in \Oo(\underline{a};a), y \in \Oo(\underline{a'};a_i)$ and $z \in \Oo(\underline{a''};a'_j)$ we have
%\begin{align*}
%(x^\psi \circ_i y^\psi) \circ_{i-1+j} z^\psi &= (-1)^{\eps + \eps'} ((x\circ_i y) \circ_{i-1+j} z)^\psi  = (-1)^{\eps+\eps'} (x\circ_i (y\circ_j z))^\psi = x^\psi \circ_i (y^\psi \circ_j z^\psi)
%\end{align*}
\begin{align*}
(x^\psi \circ_i y^\psi) \circ_{i-1+j} z^\psi = &(-1)^{  \psi(\underline{a'};a_i) \sum_{s<i}\psi(a_s) + \psi(\underline{a''};a_j')(\sum_{s<j} \psi(a_s') + \sum_{s<i} \psi(a_s))}\\
&  (-1)^{ |y|\sum_{s\neq i} \psi(a_s) + |z|(\sum_{s\neq i}\psi(a_s) + \sum_{s \neq j} \psi(a'_s))} ((x\circ_i y) \circ_{i-1+j} z)^\psi 
\end{align*}
and 
\begin{align*}
x^\psi \circ_i (y^\psi \circ_j z^\psi) = &(-1)^{  \psi(\underline{a''};a_j') \sum_{s<j} \psi(a_j') + (\psi(\underline{a''};a_j') + \psi(\underline{a'};a_i)) \sum_{s<i}\psi(a_s)  } \\
& (-1)^{ |y\circ_j z|\sum_{s\neq i} \psi(a_s) + |z| \sum_{s \neq j} \psi(a_s') } (x\circ_i (y\circ_j z))^\psi
\end{align*}
%where 
%\begin{align*}
%(-1)^{  \psi(\underline{a'};a_i) \sum_{s<i}\psi(a_s) + \psi(\underline{a''};a_j')(\sum_{s<j} \psi(a_s') + \sum_{s<i} \psi(a_s)) } &= (-1)^\eps =  (-1)^{  \psi(\underline{a''};a_j') \sum_{s<j} \psi(a_j') + (\psi(\underline{a''};a_j') + \psi(\underline{a'};a_i)) \sum_{s<i}\psi(a_s)  } \\
%(-1)^{ |y|\sum_{s\neq i} \psi(a_s) + |z|(\sum_{s\neq i}\psi(a_s) + \sum_{s \neq j} \psi(a'_s))} &= (-1)^{ \eps'} = (-1)^{ |y\circ_j z|\sum_{s\neq i} \psi(a_s) + |z| \sum_{s \neq j} \psi(a_s') }
%\end{align*}
As $((x\circ_i y) \circ_{i-1+j} z)^\psi  = (x\circ_i (y\circ_j z))^\psi$, it suffices to verify that the above signs are equal. They are equal as $\psi(\underline{a'};a_i) + \psi(\underline{a''};a'_j) = \psi(\underline{a'} \circ_j \underline{a''} ;a_i)$. 

For the second axiom we compute for $x\in \Oo(\underline{a};a), y \in \Oo(\underline{a'};a_i)$ and $z \in \Oo(\underline{a''};a_k)$ that
\begin{align*}
(x^\psi \circ_i y^\psi) \circ_{k-1+n'} z^\psi = & (-1)^{ \psi(\underline{a''};a_k) (\sum_{s<k,\neq i} \psi(a_s) + \sum_{s=1}^{n'} \psi(a'_s))  + \psi(\underline{a'};a_i) \sum_{s<i} \psi(a_s) } \\
& (-1)^{ |y|\sum_{s \neq i} \psi(a_s) + |z| (\sum_{s \neq i,k} \psi(a_s) + \sum_{s=1}^{n'} \psi(a'_s))} ((x\circ_i y) \circ_{k-1+n'} z)^\psi
\end{align*}
and 
\begin{align*}
(x^\psi \circ_k z^\psi) \circ_i y^\psi =& (-1)^{ \psi(\underline{a'};a_i)  \sum_{s<i} \psi(a_s) + \psi( \underline{a''} ; a_k) \sum_{s<k} \psi(a_s) } \\
&(-1)^{ |z| \sum_{s \neq k} \psi(a_s) + |y| (\sum_{s\neq i,k} \psi(a_s) +  \sum_{s=1}^{n''} \psi(a_s'')) } ((x\circ_k z) \circ_i y)^\psi
\end{align*}
As $((x\circ_i y) \circ_{k-1+n'} z)^\psi = (-1)^{|y||z|} ((x\circ_k z) \circ_i y)^\psi$, it suffices to verify that the above signs are equal up to $(-1)^{|y| \psi(\underline{a''};a_k) + |z| \psi(\underline{a'};a_i) +  \psi(\underline{a''};a_k)\psi(\underline{a'};a_i)}$. This is a simple computation.

Verifying the unit is trivial and it is moreover clear that it is equivariant with respect to the defined symmetric action.
\end{proof}

\subsection{Totalised algebra}\label{subapptotalg}

Given $V$ an algebra over $\Oo$ with algebra map 
$$f: \Oo \longrightarrow \End(V)$$
We define an induced map of graded operads 
$$f^\psi: \prod_\psi \Oo \longrightarrow \End(\prod_\psi V)$$
by setting
$$f^\psi(x^\psi) := f(x)^\psi$$
for $x \in \Oo$. This defines the map locally, however, working over the product, we have to take into account a finiteness condition. Note however that if we replace $\prod$ by $\bigoplus$, this is always satisfied. 

\begin{lemma}\label{compilatenwelldefined}
Given
\begin{itemize}
\item  a natural number $n \in \N$,
\item an element $x=(x_{\underline{a};a}^\psi)_{\underline{a;a}} \in \prod_\psi \Oo$ of constant degree,
\item  elements $v_1 = (v_{1,a}^\psi)_a,\ldots,v_n = (v_{n,a}^\psi)_a$ in $\prod_\psi V$ of constant degree,
\item and a colour $b \in A$.
\end{itemize}
 If the sum
$$\sum_{\underline{a};b} f(x)^\psi(v_{1,a_1}^\psi, \ldots,v_{n,a_n}^\psi)$$
is finite, then $f^\psi$ is well-defined.
\end{lemma}
\begin{proof}
This is immediate from the definition of $f^\psi$.
\end{proof}

\begin{cor}\label{algebralift}
Let $f:\Oo \longrightarrow \End(V)$ be a $\Oo$-algebra-structure on $V$ and $f^\psi$ well-defined, then $f^\psi$ is a morphism of graded operads.
\end{cor}
\begin{proof}
As $f$ is of degree $0$, we have that $f^\psi$ also has degree $0$. As the signs of $\prod_\psi \Oo$ correspond to the signs in $\End(\prod_\psi V)$. The same holds for the symmetric action. We compute the composition as an example: for $x \in \Oo(\underline{a};a)$ and $x' \in \Oo(\underline{a'};a_i)$, we have
\begin{align*}
f^\psi (x^\psi) \circ_i f^\psi (x^{'\psi}) = f(x)^\psi \circ_i f(x')^\psi &= (-1)^{|f(x')| \sum_{j\neq i} \psi(a_j) + \psi(\underline{a'};a_i) \sum_{j<i} \psi(a_j)} (f(x) \circ_i f(x'))^\psi\\
&=  (-1)^{|f(x')| \sum_{j\neq i} \psi(a_j) + \psi(\underline{a'};a_i) \sum_{j<i} \psi(a_j)} (f(x\circ_i x'))^\psi \\
&=  (-1)^{|x'| \sum_{j\neq i} \psi(a_j) + \psi(\underline{a'};a_i) \sum_{j<i} \psi(a_j)} f^\psi(x\circ_i x')^\psi \\
&=  f^\psi (x^\psi \circ_i x^{'\psi}).
\end{align*}

\end{proof}

Following lemma \ref{signoperationexplicit}, for $x \in \Oo(\underline{a};a)$ and $v_i \in V(a_i)$, we have 
$$f^\psi (x^\psi)(v_1^\psi,\ldots,v_n^\psi) := (-1)^{ \frac{\psi(\underline{a};a)(\psi(\underline{a};a)+1)}{2} +  \psi(a) \psi(\underline{a};a) + \sum_{i=1}^n|v_i^\psi| \sum_{j>i} \psi(a_j) }  f(x)(v_1,\ldots,v_n)^\psi $$

\def\cprime{$'$}
\providecommand{\bysame}{\leavevmode\hbox to3em{\hrulefill}\thinspace}
\bibliography{Bibfile}
\bibliographystyle{amsplain}

\end{document}